\documentclass[12pt,a4paper]{article}

\usepackage{tensor}

\usepackage{etoolbox}
\usepackage{mathtools}

\BeforeBeginEnvironment{cases}{\displaystyle}
\BeforeBeginEnvironment{cases}{}

\usepackage{amssymb,amsmath,amsthm,epsfig}
\numberwithin{equation}{section}
\numberwithin{figure}{section}
\numberwithin{table}{section}

\DeclareMathOperator\supp{supp}
\usepackage{mathrsfs}

\usepackage[table,xcdraw]{xcolor}

\usepackage{latexsym, enumerate}
\usepackage{eepic}
\usepackage{epic}
\usepackage{color}
\usepackage{ifpdf}

\usepackage[numbers,sort&compress]{natbib}
\usepackage{natbib}

\usepackage{verbatim}
\usepackage{titletoc}

\usepackage{dsfont}
\usepackage{multirow}
\usepackage{makecell}
\usepackage{bm}
\usepackage{epstopdf}
\usepackage{graphicx}
\usepackage{subfig}

\usepackage{tikz}

\graphicspath{{Figures/}{Figures/figure/}{Figures/adrIP/}{Figures/twoDimDatasets/}{Figures/locatesource/}{Figures/inverseopenarc/}{Figures/beam/}}

\usepackage{hyperref}
\hypersetup{hypertex=true,
	colorlinks=true,
	linkcolor=blue,
	filecolor=blue,
	urlcolor=blue,
	citecolor=blue,
	bookmarksopen=true,
}
\usepackage{soul}
\usepackage{xcolor}
\usepackage{color}
\sethlcolor{green}
\soulregister\cite7 
\soulregister\citep7 
\soulregister\citet7 
\soulregister\ref7 
\soulregister\pageref7 

\usepackage[toc,page]{appendix}  

\usepackage{booktabs}
\usepackage{caption}

\usepackage{algorithm}
\usepackage{algorithmicx}
\usepackage{algpseudocode}

\theoremstyle{plain}
\newtheorem{lem}{Lemma}[section]
\newcommand{\Thmref}[1]{Theorem \ref{#1}}
\newtheorem{thm}[lem]{Theorem}

\theoremstyle{definition}

\theoremstyle{remark}

\newcommand{\norm}[1]{\lVert #1 \rVert}

\begin{document}
	\title{\LARGE\bf  Vanishing Clutter: Fast and Accurate Shape Imaging via Active Cloaking}
	
	\author{
		Haoqiang Xiao\thanks
		{School of Mathematics, Hunan University, Changsha 410082, China.
			Email: xiaohaoqiang@hnu.edu.cn}
		\and
		Guang-Hui Zheng\thanks
		{School of Mathematics, Hunan Provincial Key Laboratory of Intelligent Information Processing and Applied Mathematics, Hunan University, Changsha 410082, China.
			Email: zhenggh2012@hnu.edu.cn (Corresponding author)}
	}
	
	\date{}
	\maketitle
\begin{center}{\bf ABSTRACT}
\end{center}\smallskip
Imaging a target sample in near-field scanning optical microscopy (NSOM) is fundamentally limited by measurement artifacts and data contamination from multiple probe–sample scattering. The probe, essential for subwavelength resolution, inherently perturbs the local field, degrading the signal-to-noise ratio and rendering the inverse problem for quantitative shape reconstruction highly ill‑posed. We first establish the well‑posedness of the corresponding forward model, providing a rigorous foundation for subsequent imaging. We then re‑purpose active cloaking—conventionally the antagonist of imaging—as an enabling mechanism to eliminate probe‑induced interference. Rather than directly reconstructing the sample from corrupted data, we actively cloak the probe by formulating an optimal control problem and prove the existence and stability of its minimizers. Leveraging the theory of localized anomalous resonance in layered plasmonic structures, we derive an exact closed‑form minimizer, thereby circumventing the computationally prohibitive iterative solution of the optimal control problem. The resulting cloaking‑driven interference removal yields a virtually probe‑free measurement environment, enabling fast, artifact‑free shape reconstruction. Extensive numerical experiments demonstrate accurate shape reconstruction and dramatic acceleration—often by orders of magnitude—over conventional iterative methods applied directly to probe‑contaminated data without cloaking‑based preprocessing, validating the robustness and transformative potential of the proposed approach for high‑fidelity subwavelength imaging.


\section{Introduction}
Near-field scanning optical microscopy (NSOM) has emerged as a cornerstone technique for subwavelength imaging, overcoming the diffraction limit of far-field optics. However, a fundamental challenge persists: the very probe that enables the capture of highly localized electromagnetic fields intrinsically and severely perturbs them. As demonstrated by Alù and Engheta \cite{Alu2010} and Bernal Arango et al. \cite{BernalArango2022}, the proximity of a bare near-field probe to a sample can drastically distort the local field distribution—standing waves form, surface plasmon polaritons are scattered, and the measured signal no longer faithfully represents the isolated sample. This probe-sample coupling introduces spurious “clutter” or artifacts into the measurement data, which contaminates the imaging process, reduces the effective signal-to-noise ratio, and fundamentally limits the achievable resolution and quantitative accuracy. In applications employing multiple probes, such as dual-scanning NSOM (DSNOM) for independent excitation and collection \cite{Fujimoto2012} or combined atomic force and optical microscopy \cite{Berezin2014}, these mutual scattering interactions become even more complex and debilitating. Consequently, the reliable recovery of a sample's true shape and material parameters from such corrupted near-field data is a highly ill-posed inverse problem, where small measurement errors can lead to disproportionately large artifacts in the reconstructed image.

To make our discussion concrete, we now introduce a precise mathematical formulation that models this multilayer coated probe near-field imaging scenario. Building upon the concept of a plasmonic cloaking layer for a single tip \cite{Alu2010,Bilotti2011} and generalizing it to multi-layered configurations, we represent the imaging apparatus as a collection of cylindrical structures. As schematically depicted in Fig. 1, the configuration consists of a penetrable sample \( D_1 \) to be imaged, a probe \( D_2 \), and a functional cloaking coating shell \( D_3 \) such that \( D_2 \subset D_3 \) and \( D_1 \cap D_3 = \emptyset \). While real probes and samples have complex three-dimensional geometries, in this work, the domains \( D_1 \), \( D_2 \), and \( D_3 \) are modeled as cylindrical inclusions, making them effectively two-dimensional domains assumed to be of class \( C^{1,\eta} \) for some \( 0<\eta<1 \). This idealization is perfectly suited for analyzing layered core-shell nanoprobes and their interactions with a sample target. The sample \( D_1 \) and the inner probe core \( D_2 \) are made of normal dielectric materials with positive permittivities \( \varepsilon_1 > 0 \) and \( \varepsilon_2 > 0 \), respectively. The background medium has a constant permittivity \( \varepsilon_m > 0 \). The cloaking layer \( D_3 \setminus \overline{D_2} \) is filled with a metamaterial whose frequency-dispersive permittivity is \( \varepsilon_3 = -s + i\delta \), where \( s>0 \) and the loss parameter \( \delta>0 \) serve as a Drude-like model for a plasmonic shell in the visible or infrared regime. This modeling directly connects to the layered metal-dielectric structures used for near-perfect scattering suppression \cite{Paria2019,Shalin2015}. Collectively, the permittivity distribution is
\[ \varepsilon(x) = \varepsilon_1 \chi(D_1) + \varepsilon_2 \chi(D_2) + \varepsilon_3 \chi(D_3 \setminus \overline{D_2}) + \varepsilon_m \chi(\mathbb{R}^2 \setminus (\overline{D_1 \cup D_3})). \]
Under a quasi-static illumination by an incident potential \( u^i(x) = x \cdot d \) with a unit vector \( d \), the total electric potential \( u \) satisfies the transmission problem
\begin{align}
\begin{cases}
\nabla \cdot (\varepsilon \nabla u) = 0 \quad & \text{in } \mathbb{R}^{2} \setminus (\partial D_{1} \cup \partial D_{2} \cup \partial D_{3}), \\
u|_{+} = u|_{-} \quad & \text{on } \partial D_{1} \cup \partial D_{2} \cup \partial D_{3}, \\
\varepsilon_{m} \frac{\partial u}{\partial \nu} \big|_{+} = \varepsilon_{1} \frac{\partial u}{\partial \nu} \big|_{-} \quad & \text{on } \partial D_{1}, \\
\varepsilon_{3} \frac{\partial u}{\partial \nu} \big|_{+} = \varepsilon_{2} \frac{\partial u}{\partial \nu} \big|_{-} \quad & \text{on } \partial D_{2}, \\
\varepsilon_{m} \frac{\partial u}{\partial \nu} \big|_{+} = \varepsilon_{3} \frac{\partial u}{\partial \nu} \big|_{-} \quad & \text{on } \partial D_{3}, \\
(u - u^{i})(x) = O(|x|^{-1}) \quad & \text{as } |x| \to \infty,
\end{cases} \label{Model}
\end{align}
where \( \nu \) denotes the outward unit normal. This configuration effectively serves as a mathematical surrogate for a multi-layered near-field probe, where the role of the “clutter” is played by the unwanted scattering signature of the coated probe itself ($D_3$).

     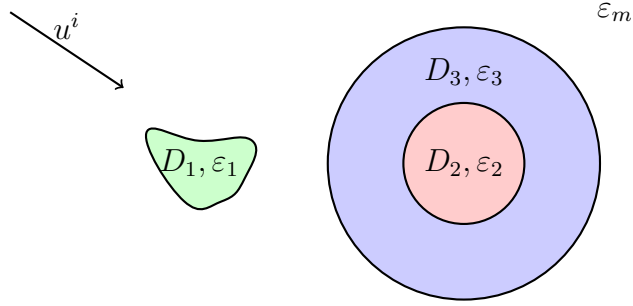
\begin{figure}[ht]
     	\centering
     	\begin{tikzpicture}[
     		region/.style={draw, thick, fill=#1!20}
     		]
     		
     		\filldraw[region=blue] (2,0) circle (1.8);
     		
     		\filldraw[region=red] (2,0) circle (0.8);
     		

          \filldraw[region=green] plot [smooth cycle, tension=0.8] coordinates {
         (-2.2,0.4) (-1.5,0.3) (-0.8,0.3)
         (-0.9,-0.2) (-1.2,-0.5) (-1.7,-0.5)
          };

     		\node at (-1.5,0) {$D_{1},\varepsilon_1$};  
     		\node at (2,0) {$D_{2},\varepsilon_2$};      
     		\node at (2,1.2) {$D_{3},\varepsilon_3$};    
     		
     		\node at (4,2) {$\varepsilon_m$};
     		\draw[->, thick, black] (-4,2) -- (-2.5,1) node[midway, above] {$u^i$};
     	\end{tikzpicture}
     \caption{ Scattering of an incident wave $u^{i}$ by a system of a sample($D_{1}$),the probe($D_{2}$) and the shell ($D_{3}\backslash \overline{D_{2}}$).} \label{fig1}
     \end{figure}
With this model, the forward problem is to determine the total field \( u \) from the incident field \( u^{i} \), given the geometry and permittivities of the sample and the multi-layered probe. The inverse problem of primary interest is to recover the shape and location of the unknown sample \( D_1 \) from measurements of the total field \( u \) on a remote circle \( \partial B_R \), assuming that the probe geometry and parameters are known. This inverse shape problem is severely ill-posed: small measurement errors are amplified dramatically, and the presence of the highly scattering, resonant probe shell \( D_3 \) exacerbates instability, with the probe's field leaking into and corrupting the sample's signature. Our strategy is to first actively cloak the probe assembly \( D_2 \cup D_3 \) so that its scattered field is suppressed and the measurement interference is eliminated, effectively revealing \( D_1 \) as if it were probed by an ideal, invisible sensor.

The classical framework for tackling such an inverse shape problem is built upon the theory of layer potentials and the systematic use of asymptotic expansions of scattered fields. Foundational work established a comprehensive toolkit—polarization and moment tensors—enabling high-resolution reconstruction of small inclusions from boundary measurements \cite{Ammari2007}. Subsequently, a rich collection of mathematical methods for multistatic imaging was developed \cite{Ammari2013b}, along with algorithms for the simultaneous reconstruction of shape and impedance parameters \cite{Cakoni2014} and the extraction of fine geometric details from spectral data \cite{Ruiz2018}. While these methods are effective in relatively benign environments, they are not designed to handle the severe data contamination introduced by a resonant, multi-layered probe in the near-field. Directly applying iterative optimization schemes in this context requires solving the forward problem at each step for a complex coupled system involving the sample and the probe. The resulting cost landscape is riddled with local minima caused by the probe's dominant and distorting scattered field, making convergence to the true sample shape difficult and computationally prohibitive. This fundamental limitation underscores a critical need: rather than attempting to computationally “see through” the probe's scattering interference, one should ask whether the probe’s influence can be eliminated physically from the data—before reconstruction is even attempted.

This question naturally connects to the rapidly maturing field of cloaking. Since the pioneering works on transformation optics \cite{Leonhardt2006,Pendry2006} and the fundamental studies on non-uniqueness in inverse problems \cite{Greenleaf2003,Greenleaf2009}, a variety of cloaking mechanisms have been developed, including change-of-variable cloaking \cite{Kohn2008} and acoustic counterparts \cite{Chen2007,Cummer2007,Zhao2026}. A particularly compelling mechanism is cloaking by anomalous localized resonance (ALR), first proposed \cite{Milton2006} and later placed on a rigorous spectral footing through the analysis of the Neumann–Poincaré operator \cite{Ammari2013a,Ando2016,DengLiu2024,Kohn2014}. Crucially, the ALR phenomenon is intimately tied to the geometry of layered, plasmonic structures—exactly the multi-layered core-shell geometry used to model cloaked near-field probes \cite{Alu2010,Bilotti2011}. These theoretical foundations have been extended to elastodynamic contexts \cite{Li2023} and continue to push the boundaries of bandwidth limits \cite{Hu2026}.

Yet, cloaking and imaging have historically been viewed as antagonists: cloaking is designed to conceal, whereas imaging strives to reveal. The deliberate application of cloaking to the sensor itself, as a means to purify measurements, remains an underexplored paradigm shift. This paper introduces a conceptual and mathematical framework that reverses this dichotomy. Instead of cloaking the target to hide it, we cloak the near-field probe(s)—the source of measurement clutter—rendering them invisible and free of artifacts, exactly as if the tip were not there \cite{Alu2010}. In this way, active cloaking is transformed from a concealment tool into a powerful, physics-driven preprocessor that purifies the scattering environment before imaging.

The central contributions of this work are as follows:
\begin{itemize}

    \item We subvert the intrinsic antagonism between imaging and cloaking by actively hiding the probe, and thereby first establish a rigorous analytical and computational framework for this paradigm, turning cloaking into a computationally direct tool for fast, artifact‑free near‑field shape reconstruction.

    \item We formulate the active suppression of the probe's spurious scattering signature as an optimal control problem where the control is the permittivity of the shell $D_{3}\backslash \overline{D_{2}}$. We prove the existence and stability of minimizers, providing a rigorous basis for active probe cloaking.

    \item By exploiting the spectral theory of the Neumann–Poincaré operator associated with the layered probe’s interfaces, we derive an exact closed-form minimizer of the optimal control problem under suitable conditions on the permittivity and geometry. This analytical result not only eliminates the need for any iterative solution but also unveils the intrinsic connection between two distinct cloaking strategies—one formulated via optimal control, and the other rooted in localized anomalous resonance.

    \item We establish the well-posedness of the forward scattering problem in the presence of both the sample \( D_1 \) and the coated probe \( D_3 \). We prove uniqueness, existence, and stability of the solution to \eqref{Model}, thereby providing a rigorous mathematical foundation that is a prerequisite for both the optimal control formulation and the subsequent imaging of the sample.

    \item After the probe is cloaked, the residual scattered field corresponds to that of the isolated sample \( D_1 \) in a homogeneous medium. We develop a fast, high-accuracy shape reconstruction algorithm that recovers \( D_1 \) directly from the decloaked data. Numerical experiments, performed on single, multi-layered, and multi-probe geometries, demonstrate orders-of-magnitude acceleration over conventional iterative inversion applied to contaminated data, while preserving high fidelity and robustness.
\end{itemize}
These contributions collectively establish active cloaking as a rigorous and enabling preprocessor for probe-artifact-free near-field imaging.

The remainder of this paper is organized as follows. Section 2 introduces preliminaries on layer potentials and the theory of anomalous localized resonance and cloaking. Section 3 establishes the well-posedness of the forward problem. Section 4 develops the active probe cloaking framework, presenting both the optimal control formulation and the resonance-based closed-form scheme. Section 5 is devoted to the fast shape reconstruction algorithm and numerical experiments. Finally, Section 6 concludes with a discussion of extensions.

\section{Preliminary}
\subsection{Layer potential}
	We denote by $G(x,y)$ the Green function for the Laplacian in the free space. We have
	\[G(x,y) = \frac{1}{2 \pi} \log|x-y|.\]
	\par We assume that the bounded domain D is the open complement of an unbounded domain of class of $C^{1,\eta}$ in $\mathbb{R}^{2}$ for $0<\eta<1$. The single potential $\mathcal{S}_{D}$ is given by
	\[ \mathcal{S}_{D}[\varphi](x) = \int_{\partial D} G(x,y) \varphi(y)d\sigma(y), \quad x \in \mathbb{R}^{2}.\]
	The Neumann-Poincar\'e (NP) operator $\mathcal{K}^{*}_{D}$ is defined by
	\[ \mathcal{K}^{*}_{D}[\varphi](x) = \int_{\partial D} \frac{\partial G(x,y)}{\partial \nu(x)}
	 \varphi(y)d\sigma(y), \quad x \in \partial D,\]
	 where $\nu(x)$ denotes the outward unit normal of $D$.
	 we have the following jump relations:
	 \begin{align}
	 	\mathcal{S}_{D}[\varphi]\big|_{+}(x) &= \mathcal{S}_{D}[\varphi]\big|_{-}(x), \quad x \in \partial D, \label{jumprelation1} \\
	 	\frac{\partial \mathcal{S}_{D}[\varphi] }{\partial \nu} \Big|_{\pm}(x) &= \Big(\pm \frac{1}{2}I + \mathcal{K}^{*}_{D} \Big)
	 	[\varphi](x) , \quad x \in \partial D. \label{jumprelation2}
	 \end{align}

    We denote by $H^{-1/2}(\partial D)$  the dual space of $H^{1/2}(\partial D)$. The duality pairing of $H^{-1/2}$ and $H^{1/2}$ is written as $\langle \cdot,\cdot \rangle$. Let $H^{-1/2}_{0}(\partial D)$ be the space of $\varphi \in H^{-1/2}(\partial D)$ satisfying $\langle \varphi,1 \rangle = 0$. We define $\mathcal{H}^{*}(\partial D)$ to be $H^{-1/2}_{0}(\partial D)$ equipped with the inner product given by
    \[\langle \varphi,\phi \rangle_{\mathcal{H}^{*}(\partial D)} = -\langle \varphi,
    \mathcal{S}_{D}[\phi] \rangle\] 
    for $\varphi,\phi \in H^{-1/2}_{0}(\partial D)$. The norm $\norm{\cdot}_{\mathcal{H}^{*}(\partial D)}$ induced by this inner product is equivalent to
    the $H^{-1/2}_{0}(\partial D)$ norm(see \cite{Ando2016}).

    The NP operator $\mathcal{K}^{*}_{D}$ is self-adjoint and compact in $\mathcal{H}^{*}(\partial D)$. Its spectrum is discrete and lies within the interval $(-\frac{1}{2},\frac{1}{2}]$. Furthermore, the free-space Green function $G(x,y)$ admits the following spectral expansion:
    \[G(x,y) = -\sum_{j=0}^{\infty}\mathcal{S}_{D}[\varphi_{j}](x) \mathcal{S}_{D}[\varphi_{j}](y)
    + \mathcal{S}_{D}[\varphi_{0}](x) \quad for \ x\in \mathbb{R}^{2}\backslash\overline{D}\ and \ y \in \overline{D},\]
    where $\varphi_{j},j=1,2,\cdots$, are eigenfunctions of $\mathcal{K}^{*}_{D}$ on $\mathcal{H}^{*}(\partial D)$ and $\varphi_{0}$ is an eigenfunction associated to the eigenvalue
    $\frac{1}{2}$. Then for any $\varphi \in \mathcal{H}^{*}(\partial D)$, we have
    \begin{align}
    	\mathcal{S}_{D}[\varphi](x) &= \sum_{j=1}^{\infty}\mathcal{S}_{D}[\varphi_{j}](x)
    	\langle \varphi,\varphi_{j}\rangle_{\mathcal{H}^{*}(\partial D)} +  \mathcal{S}_{D}[\varphi_{0}]\int_{\partial D} \varphi(y)d\sigma(y) \nonumber \\
    	&=\sum_{j=1}^{\infty}\mathcal{S}_{D}[\varphi_{j}](x)
    	\langle \varphi,\varphi_{j}\rangle_{\mathcal{H}^{*}(\partial D)}. \label{GreenFunctionSpectralExpan} 
    \end{align}
    For more details on the properties of the NP operator and $\mathcal{H}^{*}(\partial D)$, we
    refer to \cite{Ammari2007}.

\subsection{Anomalous localized resonance and cloaking}
The functionality of the coated probe as a cloaking device is intimately related to the phenomenon of anomalous localized resonance (ALR). To illustrate this mechanism, we temporarily ignore the sample \(D_1\) and focus on the probe–coating system $D_{3}$. For simplicity, we set the background permittivity and the core permittivity to unity, i.e., \(\varepsilon_m=\varepsilon_2=1\), and let the plasmonic shell possess a permittivity \(\varepsilon_3=-1+\mathrm{i}\delta\) with a small loss parameter \(\delta>0\). The limit \(\delta\to0\) corresponds to the ideal plasmonic case where the real part of \(\varepsilon_3\) approaches \(-1\), the negative of the surrounding permittivity.

Let \(f\in L^1(\mathbb{R}^2)\) be a source function with compact support satisfying the charge neutrality condition \(\int_{\mathbb{R}^2}f\,dx=0\). We consider the quasi-static dielectric problem
\[
\nabla\cdot(\varepsilon\nabla V_\delta)=\alpha f\quad\text{in }\mathbb{R}^2,\qquad V_\delta(x)\to0\;\text{as }|x|\to\infty,
\]
where $\alpha$ is a normalisation constant. The energy dissipated in the shell is proportional to
\[
E_\delta=\int_{D_3\setminus D_2}\delta\,|\nabla V_\delta|^2\,dx.
\]

Anomalous localized resonance occurs if, as \(\delta\to0\), the field \(V_\delta\) diverges in a region that extends beyond the shell while remaining bounded outside some sufficiently large radius. Quantitatively, we say that the source \(f\) induces ALR if, for \(\alpha=1\),
\[
E_\delta\to\infty\quad\text{and}\quad |V_\delta(x)|\le C\;\text{for }|x|>a,
\]
with constants \(C,a\) independent of \(\delta\). Under these conditions, renormalising the source to \(\widetilde{f}=f/\sqrt{E_\delta}\) (equivalently, taking \(\alpha=1/\sqrt{E_\delta}\)) yields a solution \(\widetilde{V}_\delta=V_\delta/\sqrt{E_\delta}\) that vanishes in the far field while the dissipated power remains finite. This is the essence of cloaking due to anomalous localized resonance (CALR): the source becomes invisible to an exterior observer. A weaker form, termed weak CALR, occurs when \(\limsup_{\delta\to0}E_\delta=\infty\) together with the far-field boundedness; then invisibility is achieved for a sequence of vanishing loss parameters.

The analysis of CALR for the concentric circular geometry (where \(D_2\) and \(D_3\) are disks of radii \(r_i\) and \(r_e\), respectively) reveals a critical radius
\[
r_*=\sqrt{r_i^3 r_e^{-1}}.
\]
Using layer potential techniques, the problem can be reduced to a singularly perturbed system of non-self-adjoint integral equations, which admits a symmetrisation. It can be shown that if the source \(f\) is supported inside the annular region \(B_{r_*}\setminus \overline{D}_2\) and its Newtonian potential does not extend harmonically into \(B_{r_*}\), then weak CALR takes place; a mild gap condition on the Fourier coefficients of the Newtonian potential guarantees full CALR. Conversely, sources supported outside \(B_{r_*}\) do not exhibit ALR and cannot be cloaked.

In the context of our near-field imaging setup, the sample \(D_1\) acts as a perturbation that can be represented by an effective source located in the vicinity of the probe. By appropriately designing the coating so that this effective source lies inside the critical region, the dominant scattering signature of the probe itself is strongly suppressed through CALR, while the weak signal from the sample remains detectable. This physical picture forms the basis for the ultra-low clutter imaging strategy analyzed in the following sections.
	
\section{Well-posedness of forward peroblem}
We first establish the well-posedness of the forward problem \eqref{Model}, which underlies the cloaking analysis developed in the following sections.
\subsection{Uniqueness}
For our immediate convenience we establish a uniqueness theorem for equation \eqref{Model}. Given $u^{i} =0$, we obtain that $u$ satisfies
	\begin{align}
		u(x) = O(|x|^{-1}), \quad \frac{\partial u}{\partial v} = O(|x|^{-2}),\quad &as \ |x| \to \infty. \label{SRMC}
	\end{align}
    Using the energy estimation method, we obtain the following uniqueness theorem:
    \begin{thm}\label{thm1}
    	Let $\epsilon_{1},\epsilon_{2},\epsilon_{m}>0$ and $\epsilon_{3}=-s+i\delta$ with $s>0$ and
    	$\delta>0$, the solution to equation \eqref{Model} is unique.
    \end{thm}
    \begin{proof}
    	Considering the homogeneous form of equation \eqref{Model} with $u^{i}=0$.
    	For the convenience of the proof, we rewrite the total electric potential $u$.
    	\begin{align}
    		u = u \chi(\mathbb{R}^2\backslash\overline{D_1\cup D_3}) + v_1 \chi(D_1) + v_2 \chi(D_2) + v_3 \chi(D_3\backslash\overline{D_2}).
    		\label{Model2}
    	\end{align}
        Then the transmission boundary conditions in equation \eqref{Model} become:
        \begin{align}
        	\begin{cases}
        		u = v_1,\quad \epsilon_m \partial_\nu u = \epsilon_1 \partial_\nu v_1 & \text{on } \partial D_1, \\[1mm]
        		v_3 = v_2,\quad \epsilon_3 \partial_\nu v_3 = \epsilon_2 \partial_\nu v_2 & \text{on } \partial D_2, \\[1mm]
        		u = v_3,\quad \epsilon_m \partial_\nu u = \epsilon_3 \partial_\nu v_3 & \text{on } \partial D_3.
        	\end{cases}
        	\label{TBound}
        \end{align}
    	Let $B_{R}$ be a
    	 circle of radius $R$. Assume that $(D_{1} \cup D_{3}) \subset B_{R}$. Applying Green's theorems over $B_{R}\backslash(\overline{D_{1}\cup D_{3}})$ and equation \eqref{Model} and \eqref{TBound}, we obtain
    	 \begin{align}
    	 	\int_{B_{R}\backslash(\overline{D_{1}\cup D_{3}})}|\nabla u|^{2}dx &=\int_{\partial B_{R}} \overline{u} \frac{\partial u}{\partial \nu}d\sigma - \int_{\partial D_{1}} \bar{u} \frac{\partial u}{\partial \nu}d\sigma - \int_{\partial D_{3}} \bar{u} \frac{\partial u}{\partial \nu}d\sigma \nonumber \\
    		&= \int_{\partial B_{R}} \overline{u} \frac{\partial u}{\partial \nu}d\sigma - \frac{\epsilon_{1}}{\epsilon_{m}} \int_{\partial D_{1}} \overline{v_{1}} \frac{\partial  v_{1}}{\partial \nu}d\sigma
    	    - \frac{\epsilon_{3}}{\epsilon_{m}}\int_{\partial D_{3}} \overline{v_{3}} \frac{\partial v_{3}}{\partial \nu}d\sigma \label{BRD1D2}
    	 \end{align}
     From Green's theorems over $D_{1}$, It follows that
     \begin{align}
     	\int_{D_{1}}|\nabla v_{1}|^{2}dx =  \int_{\partial D_{1}} \overline{v_{1}} \frac{\partial  v_{1}}{\partial \nu}d\sigma \label{D1}
     \end{align}
    Similarly, applying Green's theorems on $D_{3}\backslash \overline{D_{2}}$ and \eqref{TBound}, we have
    \begin{align}
    	\int_{D_{3}\backslash \overline{D_{2}}}|\nabla v_{3}|^{2}dx &=  \int_{\partial D_{3}} \overline{v_{3}}
    		\frac{\partial v_{3}}{\partial \nu}d\sigma - \int_{\partial D_{2}} \overline{v_{3}}
    		\frac{\partial v_{3}}{\partial \nu}d\sigma \nonumber \\
    	    &=\int_{\partial D_{3}} \overline{v_{3}}\frac{\partial v_{3}}{\partial \nu}d\sigma -
    	     \frac{\epsilon_{2}}{\epsilon_{3}}\int_{\partial D_{2}} \overline{v_{2}}\frac{\partial v_{2}}{\partial \nu}d\sigma \nonumber \\
    		&=\int_{\partial D_{3}} \overline{v_{3}}\frac{\partial v_{3}}{\partial \nu}d\sigma -
    		\frac{\epsilon_{2}}{\epsilon_{3}}\int_{D_{2}}|\nabla v_{2}|^{2}dx.\label{D3D2}
    \end{align}
    When we take $R$ sufficiently large, it follows form \eqref{SRMC} that
    \begin{align}
    	\int_{\partial B_{R}} \overline{u} \frac{\partial u}{\partial \nu}d\sigma=0.\label{SRMC1}
    \end{align}
    Combining \eqref{BRD1D2},\eqref{D1},\eqref{D3D2} and \eqref{SRMC1},we obtain
    \begin{align}
    	\int_{{B_{R}\backslash(\overline{D_{1}\cup D_{3}})}}|\nabla u|^{2}dx +
    	\frac{\epsilon_{1}}{\epsilon_{m}}\int_{D_{1}}|\nabla v_{1}|^{2}dx+
    	\frac{\epsilon_{2}}{\epsilon_{m}}\int_{D_{2}}|\nabla v_{2}|^{2}dx+
    	\frac{\epsilon_{3}}{\epsilon_{m}}\int_{D_{3}\backslash \overline{D_{2}}}|\nabla v_{3}|^{2}dx = 0. \label{EnergyE}
    \end{align}
    Since $\epsilon_{1},\epsilon_{2},\epsilon_{m}>0$ and $\epsilon_{3}=-s+i\delta$ with $s>0$,taking the imaginary part of \eqref{EnergyE} yields
    \begin{align}
    	-\frac{\delta}{\epsilon_{m}}\int_{D_{3}\backslash \overline{D_{2}}}|\nabla v_{3}|^{2}dx = 0 \label{V3}
    \end{align}
    Then taking the real part of \eqref{EnergyE} and using \eqref{V3}, we obtain
    \begin{align}
    	\int_{{B_{R}\backslash(\overline{D_{1}\cup D_{3}})}}|\nabla u|^{2}dx +
    	\frac{\epsilon_{1}}{\epsilon_{m}}\int_{D_{1}}|\nabla v_{1}|^{2}dx+
    	\frac{\epsilon_{2}}{\epsilon_{m}}\int_{D_{2}}|\nabla v_{2}|^{2}dx = 0. \label{uV1V2}
    \end{align}
    Therefore, we can conclude that $\nabla u = \nabla v_{1} = \nabla v_{2} = \nabla v_{3} =0$.
    Combined with the decay condition \eqref{SRMC}, it follows that $u=0$ in $\mathbb{R}^{2}$.
    \end{proof}

\subsection{Existence and stability}
We shall settle the existence problem for equation (1.1) by establishing that Fredholm's alternative applies to the problem. Following the approach in \cite{Ruiz2018}, it is first necessary to construct the Green‘s function $G_{D_{1}}(x,y)$, which denotes the total field generated by a point source located at $y$ in the presence of $D_{1}$ alone. Therefore, for $y \notin \overline{D_{1}}$, $G_{D_{1}}(x,y)$ satisfies the following equation:
    \begin{align}
    	\begin{cases}
    		\nabla \cdot \big(\epsilon_{1}\chi(D_{1}) + \epsilon_{m}\chi(\mathbb{R}^{2}\backslash\overline{D_{1}})\big) \nabla u = \delta_{y} \quad &in \ \mathbb{R}^{2}\backslash(\partial D_{1}) \\
    		u|_{+} = u|_{-}  \quad &on \ \partial D_{1}  \\
    		\epsilon_{m} \frac{\partial u}{\partial \nu} \big|_{+} = \epsilon_{1} \frac{\partial u}{\partial \nu} \big|_{-}  \quad &on \ \partial D_{1} \\
    		u(x) = O(|x|^{-1}) \quad &as \ |x| \to \infty.
    	\end{cases} \label{ModelGreen}
    \end{align}
    Moreover, $G_{D_1}$ admits the following explicit representation:
    \begin{align}
    	G_{D_{1}}(x,y) = G(x,y) + \mathcal{S}_{D_{1}}(\lambda_{1}I-\mathcal{K}^{*}_{D_{1}})^{-1}\bigg[\frac{\partial}{\partial \nu_{1}}G(\cdot,y)\bigg](x),\ for \ x \in \mathbb{R}^{2}\ and \ y \notin \overline{D_{1}},
    	\label{GD1}
    \end{align}
    where $\lambda_1 = \frac{\epsilon_1+\epsilon_m}{2(\epsilon_1-\epsilon_m)}$.

     We now proceed to construct the solution $u$ to equation \eqref{Model}. First, we define $u_{D_{1}}$ as the total field generated by the incident wave $u^{i}$ in the presence of only $D_{1}$. Additionally, we denote by $\frac{\partial}{\partial \nu_{i}}$ the outward normal derivative on the boundary $\partial D_{i}$. We can represent $u_{D_{1}}$ as
     \begin{align}
     	u_{D_{1}} (x) = u^{i}(x) + \mathcal{S}_{D_{1}}(\lambda_{1}I-\mathcal{K}^{*}_{D_{1}})^{-1}\bigg[\frac{\partial u^{i}}{\partial \nu_{1}} \bigg](x), \quad for \ x \in \mathbb{R}^{2}. \label{uD1}
     \end{align}
    In order to represent the field produced by $D_{2}$ and $D_{3}$, we define the operator
    \begin{align}
    	\mathcal{S}_{D_{i},D_{1}}[\varphi](x) = \int_{\partial D_{i}}G_{D_{1}}(x,y)\varphi(y)
    	d\sigma(y), \quad for \ i=2,3. \nonumber
    \end{align}
    From \eqref{GD1}, the above operator can be further written as
    \begin{align}
    	\mathcal{S}_{D_{i},D_{1}}[\varphi](x) = \mathcal{S}_{D_{i}}[\varphi](x) + \mathcal{S}^{1}_{D_{i},D_{1}}[\varphi](x), \nonumber
    \end{align}
    where $\mathcal{S}^{1}_{D_{i},D_{1}}$ is given by
    \begin{align}
    	\mathcal{S}^{1}_{D_{i},D_{1}}[\varphi](x) = \int_{\partial D_{i}}\mathcal{S}_{D_{1}}(\lambda_{1}I-\mathcal{K}^{*}_{D_{1}})^{-1}\bigg[\frac{\partial}{\partial \nu_{1}}G(\cdot,y)\bigg](x)\varphi(y)d\sigma(y).\nonumber
    \end{align}
    Combining \eqref{GreenFunctionSpectralExpan} with results from \cite{Ruiz2018}, for any
    $\varphi \in \mathcal{H}^{*}(\partial D_{i})$, $\mathcal{S}^{1}_{D_{i},D_{1}}[\varphi]$ can be further expressed as
    \begin{align}
    	\mathcal{S}^{1}_{D_{i},D_{1}}[\varphi](x) = \mathcal{S}_{D_{1}}(\lambda_{1}I-\mathcal{K}^{*}_{D_{1}})^{-1}
    	\frac{\partial\mathcal{S}_{D_{i}}[\varphi]}{\partial\nu_{1}}(x).\label{S1D1D2}
    \end{align}
    Thus, we arrive at the following representation for the total potential $u$
    \begin{align}
    	u(x) = u_{D_{1}}(x) + \mathcal{S}_{D_{2},D_{1}}[\varphi_{2}](x) + \mathcal{S}_{D_{3},D_{1}}[\varphi_{3}](x),\ for \ x \in \mathbb{R}^{2},
    	\label{totalpotential}
    \end{align}
    where $\varphi_{2} \in \mathcal{H}^{*}(\partial D_{2})$ and $\varphi_{3} \in \mathcal{H}^{*}(\partial D_{3}).$ Combining the boundary conditions in \eqref{Model}, \eqref{totalpotential} and jump relations \eqref{jumprelation1} and \eqref{jumprelation2}, we
    obtain the following system of equations for the density $\varphi_{2}$ and $\varphi_{3}$:
    \begin{align}
       \begin{cases}
    		\lambda_{2}\varphi_{2}+A_{11}[\varphi_{2}](x) + A_{12}[\varphi_{3}](x)= \frac{\partial u_{D_{1}}(x)}{\partial \nu_{2}},\quad for \ x\in \partial D_{2},\\
    		\lambda_{3}\varphi_{3}+A_{21}[\varphi_{2}](x) + A_{22}[\varphi_{3}](x)=\frac{\partial u_{D_{1}}(x)}{\partial \nu_{3}},\quad for \ x\in \partial D_{3},\\
        \end{cases} \label{Integral-equations}
     \end{align}
    where
    \begin{align*}
    	\begin{cases}
    		\lambda_{2} =\frac{\epsilon_{2} + \epsilon_{3}}{2(\epsilon_{2}-\epsilon_{3})},\\
    		\lambda_{3} =\frac{\epsilon_{3} + \epsilon_{m}}{2(\epsilon_{3}-\epsilon_{m})},\\
    		A_{11} =- \mathcal{K}^{*}_{D_{2}} - \frac{\partial}{\partial\nu_{2}}\mathcal{S}^{1}_{D_{2},D_{1}}, \\
    		A_{12} =-\frac{\partial}{\partial\nu_{2}} \mathcal{S}_{D_{3}} -
    		\frac{\partial}{\partial\nu_{2}}
    		\mathcal{S}^{1}_{D_{3},D_{1}},	\\
    		A_{21} =  -\frac{\partial}{\partial\nu_{3}} \mathcal{S}_{D_{2}} -
    		\frac{\partial}{\partial\nu_{3}}\mathcal{S}^{1}_{D_{2},D_{1}},	\\
    		A_{22} = -\mathcal{K}^{*}_{D_{3}} - \frac{\partial}{\partial\nu_{3}}\mathcal{S}^{1}_{D_{3},D_{1}}.
    	\end{cases}
    \end{align*}
    We now introduce operators $E,A:\mathcal{H}^{*}(\partial D_{2}) \times \mathcal{H}^{*}(\partial D_{3}) \to \mathcal{H}^{*}(\partial D_{2}) \times \mathcal{H}^{*}(\partial D_{3})$ defined by
    \begin{align*}
    	E = \begin{bmatrix} \lambda_{2}I & 0 \\ 0 & \lambda_{3}I \end{bmatrix}, \quad
    	A = \begin{bmatrix} A_{11} & A_{12} \\ A_{21} & A_{22} \end{bmatrix}.
    \end{align*}
    Setting
    \begin{align*}
    	\Phi = \begin{bmatrix} \varphi_{2} \\  \varphi_{3} \end{bmatrix}, \quad
    	g =  \begin{bmatrix}  \frac{\partial u_{D_{1}}(x)}{\partial \nu_{2}}\\  \frac{\partial u_{D_{1}}(x)}{\partial \nu_{3}} \end{bmatrix},
    \end{align*}
    then the integral equations \eqref{Integral-equations} can be written in the form
    \begin{align}
    	(E+A)\Phi=g. \label{Integral-equations-matrix}
    \end{align}	

    Under the assumptions of this paper, $\epsilon_{1},\epsilon_{2},\epsilon_{m}>0$ and $\epsilon_{3}=-s+i\delta$ with $s>0, \delta>0$, this implies $\lambda_{2},\lambda_{3} \ne 0$ and it follows that the matrix $E$ is invertible. According to \cite{Ruiz2018}, the entries of matrix $A$ are compact operators, hence, $A$ itself is compact. Thus, it suffices to prove that the homogeneous equation of \eqref{Integral-equations-matrix} has only the zero
    solution. Applying the Riesz-Fredholm theorem then yields the invertibility of $E+A$.
    \begin{thm}\label{thm2}
    	If $\epsilon_{1},\epsilon_{2},\epsilon_{m}>0$ and $\epsilon_{3}=-s+i\delta$ with $s>0, \delta>0$, then $E+A$ is invertible.
    \end{thm}
    \begin{proof}
	From the above discussion, it suffices to prove that $\Phi=0$ when $g=0$ in \eqref{Integral-equations-matrix}. We assume that
	\begin{align*}
		\omega(x) = u(x)-u_{D_{1}}(x) =
          \mathcal{S}_{D_{2},D_{1}}[\varphi_{2}](x)
         + \mathcal{S}_{D_{3},D_{1}}[\varphi_{3}](x),\ for \ x \in \mathbb{R}^{2}.
	\end{align*}
     Since $w$ satisfies the homogeneous form of equation \eqref{Model}, it follows from \Thmref{thm1} that $w=0$ in $\mathbb{R}^{2}$. Therefore
     \begin{align*}
     	\varphi_{2} &= \frac{\partial w}{\partial \nu_{2}} \bigg|_{+} - \frac{\partial w}{\partial \nu_{2}} \bigg|_{-}=0, \\
     	\varphi_{3} &= \frac{\partial w}{\partial \nu_{3}} \bigg|_{+} - \frac{\partial w}{\partial \nu_{3}} \bigg|_{-}=0,
     \end{align*}
     and the theorem is proved.
    \end{proof}
    Next, the main objective of this paper is to investigate the dependence of the solution $u$ on
     the parameter $s$ and we observe that $s$ appears only in $E$ with Lipschitz continuity, we derive a stability result for the density $\Phi$ with respect to $s$. This result is stated in the theorem below.
     \begin{thm}\label{thm3}
     Let $\epsilon_{1},\epsilon_{2},\epsilon_{m},\delta>0$. Choose parameters $\epsilon^{(1)}_{3} = -s_{1}+i\delta$ and $ \epsilon^{(2)}_{3} = -s_{2}+i\delta$, with $s_{1},s_{2}$ in a compact set
     $K\subset(0,+\infty)$. We now define $Y=\mathcal{H}^{*}(\partial D_{2}) \times \mathcal{H}^{*}(\partial D_{3})$ equipped with the norm
     \[ \norm{\Phi}_{Y} = \big(\norm{\varphi_{2}}^{2}_{\mathcal{H}^{*}(\partial D_{2})}
      + \norm{\varphi_{3}}^{2}_{\mathcal{H}^{*}(\partial D_{3})}\big)^{1/2}.\]
      Then
      \begin{align}
      	\norm{\Phi^{(1)} - \Phi^{(2)}}_{Y} \le C|s_{1}-s_{2}|.
      \end{align}\label{stability}
     \end{thm}
     \begin{proof}
     	From the derivation of \eqref{Integral-equations-matrix}, it is clear that only $E$ and $\Phi$ depend on $s$, we denote by $E^{(i)}$ and $\Phi^{(i)}$ the terms corresponding to $\epsilon^{(i)}_{3}$.  From \eqref{Integral-equations-matrix} it follows that
     	\begin{align*}
     		\begin{cases}
     			(E^{(1)} + A)\Phi^{(1)}=g,\\
     			(E^{(2)} + A)\Phi^{(2)}=g.
     		\end{cases}
     	\end{align*}
     Subtracting the two equations above and adding and subtracting $(E^{(2)} + A)\Phi^{(2)}$, we obtain
     \begin{align}\label{tuidao1}
     	(E^{(1)} + A)(\Phi^{(1)}- \Phi^{(2)}) = (E^{(2)}-E^{(1)})\Phi^{(2)}.
     \end{align}
     \Thmref{thm2} guarantees the invertibility of $E^{(i)}+A$ for our chosen parameters. Consequently, $\Phi^{(i)}$ can be expressed as $\Phi^{(i)} = (E^{(i)}+A)^{-1}g$. Therefore,
     \eqref{tuidao1} can be rewritten in the form
     \begin{align*}
     	\Phi^{(1)} - \Phi^{(2)} = (E^{(1)}+A)^{-1}(E^{(2)}-E^{(1)})(E^{(2)}+A)^{-1}g.
     \end{align*}
     Taking the norm on both sides of the above equation yields
     \begin{align}
     	\norm{\Phi^{(1)} - \Phi^{(2)}}_{Y} \le \norm{(E^{(1)}+A)^{-1}}_{\mathcal{L}(Y)}
     	\norm{E^{(2)}-E^{(1)}}_{\mathcal{L}(Y)} \norm{(E^{(2)}+A)^{-1}}_{\mathcal{L}(Y)}
     	\norm{g}_{Y}.
     \end{align}
     We observe that $E+A$ is continuously invertible with respect to $s$ on the compact set $K$.
     Consequently, its inverse $(E+A)^{-1}$ is also continuous in $s$. Since $K$ is compact, it follows that $(E+A)^{-1}$ is uniformly bounded on $K$. The operator $g$ is independent of the parameter $s$ and depends only on $\epsilon_{1}$ and the incident field $u^{i}$. Moreover, when
     $\epsilon_{1}>0$, $\norm{g}_{Y}$ is bounded. For $E^{(1)}-E^{(2)}$, we have
     \begin{align*}
     	E^{(2)}-E^{(1)} = \begin{bmatrix} (\lambda^{(2)}_{2}-\lambda^{(1)}_{2})I & 0 \\ 0 & (\lambda^{(2)}_{3}-\lambda^{(1)}_{3})I \end{bmatrix},
     \end{align*}
     where
     \begin{align*}
     	\begin{dcases}
     		\lambda^{(2)}_{1}-\lambda^{(1)}_{2} = \frac{\epsilon_{2}(\epsilon^{(2)}_{3} - \epsilon^{(1)}_{3})}{(\epsilon_{2}-\epsilon^{(1)}_{3})(\epsilon_{2}-\epsilon^{(2)}_{3})}
     		= \frac{\epsilon_{2}(s_{2}-s_{1})}
     		{(\epsilon_{2}+s_{1}-i\delta)(\epsilon_{2}+s_{2}-i\delta)},\\
     		\lambda^{(2)}_{3}-\lambda^{(1)}_{3} = \frac{\epsilon_{m}(\epsilon^{(2)}_{3} - \epsilon^{(1)}_{3})}{(\epsilon_{m}-\epsilon^{(1)}_{3})(\epsilon_{m}-\epsilon^{(2)}_{3})}
     		= \frac{\epsilon_{m}(s_{2}-s_{1})}
     		{(\epsilon_{m}+s_{1}-i\delta)(\epsilon_{m}+s_{2}-i\delta)}.\\
     	\end{dcases}
     \end{align*}
     Since $\epsilon_{2},\epsilon_{m},\delta>0$ and $s_{i}$ lying in a compact set $K\subset(0,+\infty)$, we can conclude that
     \begin{align*}
     	\norm{E^{(2)}-E^{(1)}}_{\mathcal{L}(Y)} \le C|s_{1}-s_{2}|.
     \end{align*}
     Therefore, \eqref{stability} follows from the above derivation, which completes the proof.
     \end{proof}
\section{Active cloaking}
\subsection{Optimal control scheme}
     The goal of this paper is to achieve cloaking for the core-shell probe $D_{3}$ by designing a shell with a permittivity of the form $\epsilon_{3}=-s+i\delta$. Based on the representation of the total potential $u$ given in \eqref{totalpotential}, we can decompose $u$ as
     \begin{align}
     	u(x) = u_{D_{1}}(x) + u_{D_{2},D_{3}}(x) \quad for \ x \in \mathbb{R}^{2}, \label{totalpotential_decompose}
     \end{align}
     where
     \begin{align}\label{uD2D3}
     	u_{D_{2},D_{3}}(x) = \mathcal{S}_{D_{2}}[\varphi_{2}](x) + \mathcal{S}^{1}_{D_{2},D_{1}}[\varphi_{2}](x)
       + \mathcal{S}_{D_{3}}[\varphi_{3}](x)
     	+ \mathcal{S}^{1}_{D_{3},D_{1}}[\varphi_{3}](x),\ for \ x \in \mathbb{R}^{2}.
     \end{align}
     According to \eqref{uD1},\eqref{S1D1D2} and \eqref{totalpotential},
     $u_{D_{1}}$ denotes the total potential when only $D_{1}$ is present. Meanwhile $u_{D_{2},D_{3}}$ represents the potential generated by $D_{3}$ and the field resulting from the interaction between $D_{1},D_{2}$ and $D_{3}$. Therefore, to achieve invisibility (cloaking) for $D_{3}$,  it suffices to control the real part of the permittivity $\epsilon_{3}$ such that $u_{D_{2},D_{3}}$ is minimized.

     We take a ball $B_{R}$ such that $D_{1} \subset B_{R}$ and $\partial B_{R} \cap (D_{1} \cup D_{3}) = \emptyset$. Since $\partial D_{1}\cap \partial D_{2}\cap \partial D_{3}\cap \partial B_{R} = \emptyset$, the kernel function $G(x,y)$
     of the single-layer potential operator $\mathcal{S}_{D_{i}}$ has the property that for any fixed $y\in\partial D_{i}$(where $i=1,2,3$), $G_{y}(x) = G(x,y)$ belongs to $C^{\infty}(\partial B_{R})$.  Consequently, for every $\varphi \in \mathcal{H}^{*}(\partial D_{i})$, $\mathcal{S}_{D_{i}}[\varphi]$ is bounded and belongs to $L^{2}(\partial B_{R})$.
     Therefore, we introduce the following minimization functional:
     \begin{align}\label{minimization_functional_J}
     	\mathcal{J}(s) = \frac{1}{2}\norm{u_{D_{2},D_{3}}(s,\cdot)}^{2}_{L^{2}(\partial B_{R})},
     \end{align}
     where $\epsilon_{3} = -s+i\delta$ and $s \in K \subset (0,+\infty)$.

      \begin{thm}\label{thm01}
      	There exists a minimizer $s\in K$ for minimization functional $\mathcal{J}(s)$.
      \end{thm}
      \begin{proof}
     From \eqref{stability} and \eqref{uD2D3}, we see that $u_{D_{2},D_{3}}$ is continuous with respect to $s$. Since $s$ is restricted to a compact set $K$, it follows that $\norm{u_{D_{2},D_{3}}(s,\cdot)}_{L^{2}(\partial B_{R})}$ is uniformly bounded and continuous on
     $K$. Consequently, $\mathcal{J}$ has at least one minimizer.
     \end{proof}

     Moreover, applying
      \Thmref{thm3}, we obtain that $\mathcal{J}$ is Lipschitz stable with respect to $s$. The theorem is stated as follows.
      \begin{thm}\label{thm4}
      	Let $\epsilon_{1},\epsilon_{2},\epsilon_{m},\delta>0$. Choose parameters $\epsilon^{(1)}_{3} = -s_{1}+i\delta$ and $ \epsilon^{(2)}_{3} = -s_{2}+i\delta$, with $s_{1},s_{2}$ in a compact set
      	$K\subset(0,+\infty)$, the following Lipschitz stability holds:
      	\begin{align}
      		\big|\mathcal{J}(s_{1}) - \mathcal{J}(s_{2})\big| \le C|s_{1} - s_{2}|. \label{J_stability}
      	\end{align}
      \end{thm}
      \begin{proof}
      From \eqref{uD2D3}, we denote $u_{D_{2},D_{3}}$ corresponding to $s_{i}$ by
      \begin{align*}
      	\begin{dcases}
      		u_{D_{2},D_{3}}(s_{1},x) = \mathcal{S}_{D_{2}}[\varphi^{(1)}_{2}](x) + \mathcal{S}^{1}_{D_{2},D_{1}}[\varphi^{(1)}_{2}](x) + \mathcal{S}_{D_{3}}[\varphi^{(1)}_{3}](x)
      		+ \mathcal{S}^{1}_{D_{3},D_{1}}[\varphi^{(1)}_{3}](x), \\
      		u_{D_{2},D_{3}}(s_{2},x) = \mathcal{S}_{D_{2}}[\varphi^{(2)}_{2}](x) + \mathcal{S}^{1}_{D_{2},D_{1}}[\varphi^{(2)}_{2}](x) + \mathcal{S}_{D_{3}}[\varphi^{(2)}_{3}](x)
      		+ \mathcal{S}^{1}_{D_{3},D_{1}}[\varphi^{(2)}_{3}](x).
      	\end{dcases}
      \end{align*}
     Based on the preceding discussion, we know that
     $\mathcal{S}_{D_{i}}: \mathcal{H}^{*}(\partial D_{i}) \to L^{2}(\partial B_{R})$
     is bounded. Combining this with \eqref{stability} yields
     \begin{align*}
     	\norm{u_{D_{2},D_{3}}(s_{1},\cdot)-u_{D_{2},D_{3}}(s_{2},\cdot)}_{L^{2}(\partial B_{R})} \le C|s_{1} - s_{2}|.
     \end{align*}
     Then we can obtain
     \begin{align*}
     	\big|\mathcal{J}(s_{1}) - \mathcal{J}(s_{2}) \big| &= \frac{1}{2}\Big|\norm{u_{D_{2},D_{3}}(s_{1},\cdot)}^{2}_{L^{2}(\partial B_{R})}-\norm{u_{D_{2},D_{3}}(s_{2},\cdot)}^{2}_{L^{2}(\partial B_{R})}\Big| \\
     	&\le C\Big|\norm{u_{D_{2},D_{3}}(s_{1},\cdot)}_{L^{2}(\partial B_{R})}
     	 - \norm{u_{D_{2},D_{3}}(s_{2},\cdot)}_{L^{2}(\partial B_{R})}\Big| \\
     	 &\le C\norm{u_{D_{2},D_{3}}(s_{1},\cdot)-u_{D_{2},D_{3}}(s_{2},\cdot)}_{L^{2}(\partial B_{R})}\\
     	 & \le C|s_{1} - s_{2}|.
     \end{align*}
      \end{proof}
  \begin{thm}\label{thm5}
  	Let $\{\mu_{k}\}$ be sequences such that
  	\begin{align}\label{raodong}
  	\norm{u^{(k)}_{D_{2},D_{3}}(s,\cdot) - u_{D_{2},D_{3}}(s,\cdot)}_{L^{2}(\partial B_{R})} \le \mu_{k},
  	\end{align}
   $s_{k}$ is a minimizer of \eqref{minimization_functional_J} with $u_{D_{2},D_{3}}$ replaced by
   $u^{(k)}_{D_{2},D_{3}}$. When $\mu_{k} \to 0$ as $k\to \infty$, there exists a convergent subsequence of $\{s_{k}\}$ and the limit of convergent subsequence is a minimizer of \eqref{minimization_functional_J}.
  \end{thm}
\begin{proof}
	By the definition of $s_{k}$ we have
	\begin{align}\label{tuidao2}
		\norm{u^{(k)}_{D_{2},D_{3}}(s_{k},\cdot)}^{2}_{L^{2}(\partial B_{R})} \le
		\norm{u^{(k)}_{D_{2},D_{3}}(s,\cdot)}^{2}_{L^{2}(\partial B_{R})}, \ \forall s \in K.
	\end{align}
Since $\{s_{k}\} \in K$, there exists a convergent subsequence of $\{s_{k}\}$, still denoted as $\{s_{k}\}$, such that $s_{k} \to s^{*}$ and $s^{*} \in K$.
Based on the continuity of the $L^{2}$-norm, together with \eqref{raodong} and \eqref{tuidao2}, we
can obtain
\begin{align*}
	\norm{u_{D_{2},D_{3}}(s^{*},\cdot)}^{2}_{L^{2}(\partial B_{R})} =&
	\lim_{k\to \infty}\norm{u_{D_{2},D_{3}}(s_{k},\cdot)}^{2}_{L^{2}(\partial B_{R})} \\
	=& \lim_{k\to \infty}\norm{u^{(k)}_{D_{2},D_{3}}(s_{k},\cdot)}^{2}_{L^{2}(\partial B_{R})} \\
	\le& \lim_{k\to \infty}\norm{u^{(k)}_{D_{2},D_{3}}(s,\cdot)}^{2}_{L^{2}(\partial B_{R})} \\
	=& \norm{u_{D_{2},D_{3}}(s,\cdot)}^{2}_{L^{2}(\partial B_{R})}, \ \forall s \in K.	
\end{align*}
This implies that $s^{*}$ is a minimizer of \eqref{minimization_functional_J}.
\end{proof}

\subsection{Anomalous localized resonance scheme}
In the previous sections, we have proved that the minimization functional $\mathcal{J}(s)$ admits at least one minimizer $s^{*}$ on a bounded compact set. However, the explicit form of $s^{*}$ remains unknown, and it is not clear whether $\mathcal{J}(s^{*})$ tends to zero as
$\delta \to 0$; that is, whether the invisibility of $D_{3}$ can be achieved. In this section, we will employ the technique of local anomalous resonance from \cite{Ammari2013a} to show that, with an appropriate choice of parameters, there indeed exists an $s^{*}$ that
makes $D_{3}$ invisible.

 We assume that $D_{2}$ and $D_{3}$ form a concentric annular region centered at the origin, with $D_{2} = B_{r_{2}}=\{|x|<r_{2}\}, D_{3} = B_{r_{3}}=\{|x|<r_{3}\}$ and $r_{2}<r_{3}$. Let $D_{1}$ be centered at $z_{0}$ and satisfy $D_{1} \cap B_{r^{*}} = \emptyset$, where $r^{*} = \frac{r^{3}_{3}}{r^{2}_{2}}$. The permittivities are given by $\epsilon_{2}=\epsilon_{m}=1, \epsilon_{3}=-1+i\delta$, and
$\epsilon_{1}$ is taken as a positive real number. Then, from \eqref{Integral-equations}, it yields
\begin{align}\label{gama-delta}
  \lambda_{2} = -\lambda_{3} = \gamma_{\delta} = \frac{i\delta}{2(2-i\delta)}.
\end{align}
The integral equation \eqref{Integral-equations-matrix} can be rewritten as
\begin{align}
    (\gamma_{\delta}\tilde{I} + K^{*} + S)\Phi = b,\label{The integral equation2}
\end{align}
where $\tilde{I}$ is the $2\times2$ identity matrix(with identity operators as its entries), $b = [g_1, -g_2]^{T}$ and the operators $K^{*},S:\mathcal{H}^{*}(\partial D_{2}) \times \mathcal{H}^{*}(\partial D_{3}) \to \mathcal{H}^{*}(\partial D_{2}) \times \mathcal{H}^{*}(\partial D_{3})$ are given by
\begin{align*}
    	K^{*}= \begin{bmatrix} -\mathcal{K}^{*}_{D_{2}} & -\frac{\partial}{\partial \nu_{2}}\mathcal{S}_{D_{3}}
         \\\frac{\partial}{\partial \nu_{3}}\mathcal{S}_{D_{2}} & \mathcal{K}^{*}_{D_{3}} \end{bmatrix}, \quad
    	S = \begin{bmatrix} - \frac{\partial}{\partial\nu_{2}}\mathcal{S}^{1}_{D_{2},D_{1}} &
         - \frac{\partial}{\partial\nu_{2}}\mathcal{S}^{1}_{D_{3},D_{1}}
         \\ \frac{\partial}{\partial\nu_{3}}\mathcal{S}^{1}_{D_{2},D_{1}}
          & \frac{\partial}{\partial\nu_{3}}\mathcal{S}^{1}_{D_{3},D_{1}} \end{bmatrix}.
\end{align*}
    Assuming that $D_{1}$ is located far enough from $D_{3}$, i.e., $z_{0}$ is sufficiently large, we set
    $\eta = \frac{1}{z_{0}}$,
    then we observe that $S$ is in fact a compact perturbation of $\gamma_{\delta}\tilde{I} + K^{*}$ and we obtain
    the following estimate for the operator $S$.
    \begin{lem}\label{lem1}
      Let $\epsilon_{1}>0,\epsilon_{m}>0$ and define $Y=\mathcal{H}^{*}(\partial D_{2}) \times \mathcal{H}^{*}(\partial D_{3})$,
      then
      $\norm{S}_{Y} = O(\eta^{2})$ as $\eta \to 0$.
    \end{lem}
    \begin{proof}
      We estimate the entries of $S$ individually. Since the analysis is identical for each entry, we only provide the proof for
      $S_{11}$ as an example. Fix $\varphi_{2} \in \mathcal{H}^{*}(\partial D_{2})$ and let
      \begin{align*}
       \widetilde{\varphi}_{2} = (\lambda_{1}I-\mathcal{K}^{*})^{-1}
       \Big[\frac{\partial \mathcal{S}_{D_{2}}[\varphi_{2}]}{\partial \nu_{1}}\Big].
      \end{align*}
      Given that $\mathcal{S}_{D_{2}}[\varphi_{2}]$ is harmonic in $D_{1}$, it follows from Green's identity that $\int_{\partial D_{1}}
       \frac{\partial}{\partial \nu_{1}}\mathcal{S}_{D_{2}}[\varphi_{2}]=0$. Moreover, since $\epsilon_{1}>0,\epsilon_{m}>0$, the operator $(\lambda_{1}I-\mathcal{K}^{*})^{-1}: \mathcal{H}^{*}(\partial D_{1}) \to \mathcal{H}^{*}(\partial D_{1}) $ is bounded, and hence
       $\widetilde{\varphi}_{2} \in \mathcal{H}^{*}(\partial D_{1})$. From \eqref{S1D1D2}.$\mathcal{S}_{11}$ can be written as
       \[ \frac{\partial}{\partial\nu_{2}}\mathcal{S}^{1}_{D_{2},D_{1}}[\varphi_{2}](x)=
        \frac{\partial}{\partial\nu_{2}}\mathcal{S}_{D_{1}}[\widetilde{\varphi}_{2}](x)
        =\int_{\partial D_{1}} \frac{\partial G(x,y)}{\partial \nu_{2}(y)}\widetilde{\varphi}_{2}(y)d\sigma(y), \ x\in \partial D_{2}.\]
        Therefore,since
        \[ \bigg|\frac{\partial G(x,y)}{\partial \nu_{2}(y)}\bigg| = O\bigg(\frac{1}{|x-y|}\bigg) = O(\eta),\]
        we obtain
        \[\bigg\vert\bigg\vert{\frac{\partial}{\partial\nu_{2}}\mathcal{S}_{D_{1}}[\widetilde{\varphi}_{2}]}\bigg\vert
        \bigg\vert_{\mathcal{H}^{*}(\partial D_{2})}
        \le C\eta\norm{\widetilde{\varphi}_{2}}_{\mathcal{H}^{*}(\partial D_{1})}.\]
        Applying the same method again yields
        \begin{align*}
          \norm{\widetilde{\varphi}_{2}}_{\mathcal{H}^{*}(\partial D_{1})}  &=
          \bigg\vert\bigg\vert (\lambda_{1}I-\mathcal{K}^{*})^{-1}
          \Big[\frac{\partial \mathcal{S}_{D_{2}}[\varphi_{2}]}{\partial \nu_{1}}\Big]
          \bigg\vert\bigg\vert_{\mathcal{H}^{*}(\partial D_{1})}       \\
          &\le C\bigg\vert\bigg\vert \frac{\partial}{\partial \nu_{1}} \mathcal{S}_{D_{2}}[\varphi_{2}]
           \bigg\vert\bigg\vert_{\mathcal{H}^{*}(\partial D_{1})} \\
          &\le C\eta\norm{\varphi_{2}}_{\mathcal{H}^{*}(\partial D_{2})}.
        \end{align*}
        Thus we obtain $\norm{\mathcal{S}_{11}}_{\mathcal{H}^{*}(\partial D_{2})} = O(\eta^{2})$. By the same argument,
        we find that the operator norms of all entries of $S$ are $O(\eta^{2})$, hence $\norm{S}_{Y} = O(\eta^{2})$.
        This completes the proof.
    \end{proof}

    Next, we turn to the analysis of the properties of $\gamma_{\delta}\tilde{I} + K^{*}$. We consider the following equation
    \begin{align}
    	\nabla \cdot \big(\epsilon_{2}\chi(D_{2}) + \epsilon_{3}\chi(D_{3}\backslash \overline{D_{2}})
        +\epsilon_{m}\chi(\mathbb{R}^{2}\backslash\overline{D_{3}})\big) \nabla u = f
         \quad &in \ \mathbb{R}^{2}\backslash(\partial D_{2} \cup \partial D_{3}), \label{equf} 
    \end{align}
    where the source $f$ is supported in $\mathbb{R}^{2}\backslash\overline{D_{3}}$. Then the solution to the equation can be expressed as
    \begin{align}
    	u(x) = F(x) + \mathcal{S}_{D_{2}}[\varphi_{2}](x) + \mathcal{S}_{D_{3}}[\varphi_{3}](x)
    	\label{Ntotalpotential} 
    \end{align}
    where
    \begin{align}
    	F(x) = \int_{\mathbb{R}^{2}}G(x-y)f(y)dy, \quad x\in \mathbb{R}^{2}. \label{Nwpoential}
    \end{align}
   Since $u$ satisfies the transmission boundary conditions on $\partial D_{2}$
   and $\partial D_{3}$, hence the potential $\Phi$ satisfies the equation:
   \begin{align}
   	(\gamma_{\delta}\tilde{I} + K^{*})\Phi = h, \label{integral-equation-NONS}
   \end{align}
    where
    \[h = \begin{bmatrix} \frac{\partial F}{\partial \nu_{2}} \\  -\frac{\partial F}{\partial \nu_{3}} \end{bmatrix}.\]
    Since $D_{2}$ and $D_{3}$ are concentric disks, $\Phi$ can be expressed as
    \[ \Phi=\sum_{n \neq 0}
        \begin{bmatrix}
        	\varphi^{(n)}_{2} \\
        	\varphi^{(n)}_{3}
        \end{bmatrix}
    	e^{in\theta}.
    \]
    Here, $\varphi^{(n)}_{2}$ and $\varphi^{(n)}_{3}$ are the Fourier coefficients of $\varphi_{2}$ and $\varphi_{3}.$
    It is given in \cite{Ammari2013a} that
    \[ K^{*}\Phi=\sum_{n \neq 0}
    \begin{bmatrix}
    	\frac{\rho^{|n|-1}}{2}\varphi^{(n)}_{2} \\
    	\frac{\rho^{|n|+1}}{2}\varphi^{(n)}_{3}
    \end{bmatrix}
    e^{in\theta}.
    \]
    where
    \[ \rho = \frac{r_{2}}{r_{3}}.\]
    Given that the support of $f$ lies in $\mathbb{R}^{2}\backslash\overline{D_{3}}$, $F$
    is harmonic in $D_{3}$. Therefore, $F$ can be expanded as
    \begin{align} \label{Fourier-NWpotential}
    	F(x)=c-\sum_{n \neq 0}\frac{h^{(n)}}{|n|r^{|n|-1}_{3}}r^{|n|}e^{in\theta},
    \end{align}
    then we have
    \[ h=\sum_{n \neq 0}
    \begin{bmatrix}
    	h^{(n)} \\
    	-h^{(n)}\rho^{|n|-1}
    \end{bmatrix}
    e^{in\theta}.
    \]
    Hence, we obtain the solution $\Phi=[\varphi_{2}, \varphi_{3}]^{T}$ to \eqref{integral-equation-NONS}
    as
    \begin{align}
    	\begin{dcases}
    		\varphi_{2} = -2\sum\limits_{n \neq 0}  \frac{(2\gamma_{\delta}+1)\rho^{|n|-1}h^{(n)}}
    		{4\gamma^{2}_{\delta}-\rho^{2|n|}}e^{in\theta}, \\
    		\varphi_{3} = 2\sum\limits_{n \neq 0}\frac{(2\gamma_{\delta}+\rho^{2|n|})h^{(n)}}
    		{4\gamma^{2}_{\delta}-\rho^{2|n|}}e^{in\theta}.
    	\end{dcases} \label{solution-potential}
    \end{align}
    Using the conclusion in \cite{Ammari2013a}, it can be shown that
    \begin{align}\label{solution-S2S3}
       \mathcal{S}_{D_{2}}[\varphi_{2}](x) + \mathcal{S}_{D_{3}}[\varphi_{3}](x) =
       \sum_{n\ne 0} \frac{2(r^{2|n|}_{2}-r^{2|n|}_{3})\gamma_{\delta}}
       {|n|r^{|n|-1}_{3}(4\gamma^{2}_{\delta}-\rho^{2|n|})}\frac{h^{(n)}}{r^{|n|}}e^{in\theta}, \quad r=|x|>r_{3}.
    \end{align}

    For the solution \eqref{solution-potential} to integral system \eqref{integral-equation-NONS}, we have
    a crucial observation, which is stated in the following lemma. This will help us analyze the limiting properties
    of the solution to the original problem \eqref{The integral equation2}.
    \begin{lem} \label{lem2} 
      Let $f$ be a source function supported in $\mathbb{R}^{2}\backslash\overline{B_{r^{*}}}$ where
      $r^{*}=\frac{r_{3}}{\rho^{2}}$ and let $F$ be the Newtonian
      potential of $f$. Let $\Phi_{\delta}$ denote the solution to \eqref{integral-equation-NONS},
      then it follows that
      \[\lim_{\delta \to 0}|\Phi_{\delta}| \le C,\]
      where $C$ is a constant independent of $\delta$.
    \end{lem}
    \begin{proof}
      Since $\supp(f) \subset \mathbb{R}^{2}\backslash\overline{B_{r^{*}}}$, then $F$ is a harmonic function
      in a neighborhood of $\overline{B_{r^{*}}}$ and the power series of $F$, which is given by \eqref{Fourier-NWpotential},
      converges for $r<r^{*} + 2\epsilon $ for some $\epsilon>0$. Therefore, we can obtain
      \begin{align}\label{harmonic-series-converge}
        \frac{h^{(n)}}{|n|r^{|n|-1}_{3}} \le C\frac{1}{(r^{*}+\epsilon)^{|n|}}.
      \end{align}
      For sufficiently small $\delta$, together with \eqref{gama-delta}, we have
      \begin{align}\label{lem2-tuidao-estmi}
        |4\gamma^{2}_{\delta}-\rho^{2|n|}| = \bigg|\frac{-4\delta^{2}(4-\delta^{2}+2i\delta)}{(4+\delta^{2})^{2}}-\rho^{2|n|}\bigg|
        > \rho^{2|n|}.
      \end{align}
      Combining \eqref{solution-potential},\eqref{harmonic-series-converge} and \eqref{lem2-tuidao-estmi},it follows that
     \begin{align*}
     	|\varphi_{2}|&\le 2\sum\limits_{n \neq 0}  \frac{|(2\gamma_{\delta}+1)\rho^{|n|-1}h^{(n)}|}
     	{|4\gamma^{2}_{\delta}-\rho^{2|n|}|} \le
     	\sum\limits_{n \neq 0} \frac{4C|n|}{r_{2}}
     	\bigg(\frac{r^{*}}{r^{*}+\epsilon}\bigg)^{|n|}< \infty. \\
     	 |\varphi_{3}|&\le 2\sum\limits_{n \neq 0}\frac{|(2\gamma_{\delta}+\rho^{2|n|})h^{(n)}|}
     	 {|4\gamma^{2}_{\delta}-\rho^{2|n|}|} \le
     	 \sum\limits_{n \neq 0} \frac{4C|n|}{r_{3}}
     	 \bigg(\frac{r^{*}}{r^{*}+\epsilon}\bigg)^{|n|}< \infty.
     \end{align*}
     This completes the proof.
    \end{proof}
     Lemma~\ref{lem2} shows that for the integral equation \eqref{integral-equation-NONS}, if $F$
     is harmonic in $B_{r_{0}}$ $(r_{0}>r^{*})$, then the solution $\Phi$ is uniformly bounded as
     $\delta \to 0$.

     We now return to the original integral equation \eqref{The integral equation2}, we next employ an asymptotic expansion technique. This allows us to transform the problem of solving it into solving an infinite sequence of integral equations \eqref{integral-equation-NONS}, through which we analyze the limiting properties of the solution. As in lemma~\ref{lem2}, we start with the source term. The source in the integral equation \eqref{The integral equation2} is $u_{D_{1}}$. From \cite{Ammari2007}, $u_{D_{1}}$ can be expanded as
     \begin{align}\label{uD1-GCPT}
       u_{D_{1}}(x) = u^{i}(x) + \sum_{|\alpha|,|\beta| \ge 1} \frac{1}{\alpha ! \beta !}\partial^{\alpha}u^{i}(z_{0})
       M_{\alpha \beta}(\lambda_{1},D_{1})\partial^{\beta}G(x,z_{0}), \quad |x-z_{0}| \to \infty,
     \end{align}
     where $M_{\alpha \beta}(\lambda_{1},D_{1})$ is given by
     \[ M_{\alpha \beta}(\lambda_{1},D_{1}) = \int_{\partial D_{1}} (y-z_{0})^{\beta}(\lambda_{1}I-\mathcal{K}^{*}_{D_{1}})^{-1}
     \Big[\frac{\partial (x-z_{0})^{\alpha}}{\partial \nu}\Big](y)d\sigma(y), \quad \alpha,\beta \in \mathbb{N}^{2}.\]
      Let $(r_{x},\theta_{x})$ be the polar coordinates of $x$. Acooriding to \cite{Ammari2013b}, for $|x| < |z_{0}|$,
      \begin{align*}
        G(x,z_{0}) = \sum_{m=0}^{\infty} \frac{-1}{2\pi  m} \frac{\cos(m\theta_{z_{0}}-m\theta_{x})r^{m}_{x}}{r^{m}_{z_{0}}}
        = \sum_{m=0}^{\infty} \eta^{m}H_{m}(x),
      \end{align*}
      where $H_{m}(x)$ is an $m$-th order homogeneous harmonic polynomial defined on $B_{r}$, with $r<\frac{1}{\eta}$. Therefore, we have
      \begin{align*}
      \partial^{\beta}G(x,z_{0}) = \sum_{m=|\beta|}^{\infty} \eta^{m}\partial^{\beta}H_{m}.
      \end{align*}
       Since $u^{i} = x\cdot d$, only the $|\alpha|=1$ in the summation of \eqref{uD1-GCPT} are nonzero. Let $C_{\beta}$ denote
       the constants in this summation and combine the harmonic polynomials of the same order. Then \eqref{uD1-GCPT} becomes
       \begin{align}\label{ASY-uD1-GCPT}
         u_{D_{1}} = x\cdot d + \sum_{|\beta|=1}^{\infty}C_{\beta}\sum_{m=|\beta|}^{\infty} \eta^{m}\partial^{\beta}H_{m}(x)
         = x\cdot d + \sum_{m=1}^{\infty}\eta^{m} \widetilde{H}_{m},
       \end{align}
       where
       \[\widetilde{H}_{m} = \sum_{|\beta| = 1}^{m}C_{\beta}\partial^{\beta}H_{m}.\]
   Consequently, the right-hand side $b$ of the integral equation \eqref{The integral equation2}
   can be expanded as
   \begin{align} \label{bbbb}
   	b= \begin{bmatrix}
   		\frac{ \partial u_{D_{1}}}{\partial \nu_{2}} \\
   		-\frac{ \partial u_{D_{1}}}{\partial \nu_{2}}
   	\end{bmatrix}
   =\begin{bmatrix}
   	d\cdot \nu_{2} \\
   	-d\cdot \nu_{3} \\
   \end{bmatrix}
   +\sum_{m=1}^{\infty}\eta^{m}  \begin{bmatrix}
   	\frac{\partial \widetilde{H}_{m} }{\partial \nu_{2}} \\
   	-\frac{\partial \widetilde{H}_{m} }{\partial \nu_{3}}
   \end{bmatrix}
   =b_{0} + \sum_{m=2}^{\infty}\eta^{m}b_{m}.
   \end{align}
    We also perform an expansion for the solution $\Phi$ of \eqref{The integral equation2}
    as follows
    \begin{align}\label{ASY-Phi}
    	\Phi = \sum_{m=0}^{\infty} \eta^{m}\Phi_{m},
    \end{align}
    where the unknown coefficients $\Phi_{m}$ remain to be determined. Substituting \eqref{bbbb} and
    \eqref{ASY-Phi} into the integral equation \eqref{The integral equation2} and
     using Lemma~\ref{lem1}, which gives $\norm{S}_{Y}=O(\eta^{2})$, then the zeroth-order term
     $\Phi_{0}$ satisfies
     \begin{align*}
     	(\gamma_{\delta}\tilde{I} + K^{*})\Phi_{0} = b_{0},
     \end{align*}
    because $u^{i}$ is harmonic in $\mathbb{R}^{2}$, Lemma~\ref{lem2} ensures that $\Phi_{0}$
    remains uniformly bounded as $\delta \to 0$. Furthermore, for $d=(\cos(\theta_{d}),\sin(\theta_{d}))$, it follows from \eqref{solution-potential} that $\Phi_{0}=(\Phi_{0}(1),\Phi_{0}(2))$ can be explicitly solved as
    \begin{align*}
    	\begin{dcases}
    		\Phi_{0}(1) = \frac{4\gamma_{\delta}+2}
    		{4\gamma^{2}_{\delta}-\rho^{2}}\cos(\theta-\theta_{d}), \\
    		\Phi_{0}(2) = \frac{-4\gamma_{\delta}-2\rho^{2}}
    		{4\gamma^{2}_{\delta}-\rho^{2}}\cos(\theta-\theta_{d}).
    	\end{dcases}
    \end{align*}
    For the first-order term $\Phi_{1}$, since $b_{1}=0$, it follows from \eqref{solution-potential}
    that $\Phi_{1}=0$. Defining $S = \eta^{2}\widetilde{S}$, then the second-order term
    $\Phi_{2}$ satisfies
    \begin{align*}
    	(\gamma_{\delta}\tilde{I} + K^{*})\Phi_{2} = -\widetilde{S}\Phi_{0} + b_{2}.
    \end{align*}
    Given that $z_{0}>r^{*}$ and sufficiently large, the entries of $S$ and $\widetilde{H}_{2}$
    are harmonic in $B_{r^{*}}$, and $\Phi_{0}$ is uniformly bounded as $\delta \to 0$. Hence,
    by Lemma~\ref{lem2}, $\Phi_{2}$ is also uniformly bounded as $\delta \to 0$. Repeating this argument, we obtain that $\Phi$ is uniformly bounded as $\delta \to 0$. We state this result as the following lemma.
    \begin{lem}
    	\label{lem3} 
    	Let $\Phi_{\delta}$ denote the solution to \eqref{The integral equation2},
    	then it holds that
    		\[\lim_{\delta \to 0}|\Phi_{\delta}| \le C,\]
    		where $C$ is a constant independent of $\delta$.
    \end{lem}
    We next investigate the behavior of $u_{D_{2},D_{3}}$ as $\delta \to 0$. We first decompose $u_{D_{2},D_{3}}$ as
    \[u_{D_{2},D_{3}}(x) = V_{\delta}(x) + T_{\delta}(x),\]
    where
    \begin{align*}
    \begin{dcases}
    	V_{\delta}(x) = \mathcal{S}_{D_{2}}[\varphi_{2}](x) + + \mathcal{S}_{D_{3}}[\varphi_{3}](x), \\
      T_{\delta}(x) =  \mathcal{S}^{1}_{D_{2},D_{1}}[\varphi_{2}](x)
     	+ \mathcal{S}^{1}_{D_{3},D_{1}}[\varphi_{3}](x).
    	\end{dcases}
    \end{align*}
    Using this notation, we have the following theorem.
    \begin{thm}\label{thm6}
      Let the center $z_{0}$ of $D_{1}$ be sufficiently far from $D_{3}$ such that $z_{0}>r^{*}$ and $D_{1}\cap B_{r^{*}}=\emptyset$. Then
       \begin{align} \label{thm-result}
       \sup_{|x| \ge r^{*}} |V_{\delta} + T_{\delta}| \to 0 \quad as \quad \delta \to 0.
       \end{align}
    \end{thm}
    \begin{proof}
      Regarding $T_{\delta} + u_{D_{1}}$ as the source term, \eqref{The integral equation2} can be written as
      \begin{align*}
        (\gamma_{\delta}\tilde{I} + K^{*})\Phi = b - S\Phi.
      \end{align*}
      Analogous to the treatment of \eqref{integral-equation-NONS}, we expand $F=T_{\delta} + u_{D_{1}}$ using Fourier series as in \eqref{Fourier-NWpotential},
      with the coefficient $h^{(n)}$ replaced by $\widetilde{h}^{(n)}$. Then from \eqref{solution-S2S3}, we obtain
      \begin{align}\label{Vdelta}
       V_{\delta}(x) =
       \sum_{n\ne 0} \frac{2(r^{2|n|}_{2}-r^{2|n|}_{3})\gamma_{\delta}}
       {|n|r^{|n|-1}_{3}(4\gamma^{2}_{\delta}-\rho^{2|n|})}\frac{\widetilde{h}^{(n)}}{r^{|n|}}e^{in\theta}, \quad r=|x|>r_{3}.
    \end{align}
    It should be emphasized that, unlike the case of \eqref{solution-S2S3}, we do not obtain an analytic expression for \eqref{Vdelta};
    the expression is only formal. The reason is that $\widetilde{h}$ is a coupling term that contains information about the density
    $\Phi$. From the previous discussion, $T_{\delta} + u_{D_{1}}$ is harmonic in $B_{r^{*}}$ and by Lemma~\ref{lem3}, $\Phi$ is
    uniformly bounded as $\delta \to 0$. Therefore, we obtain
    \begin{align}\label{thm3-tuidao}
        \frac{\widetilde{h}^{(n)}}{|n|r^{|n|-1}_{3}} \le C\frac{1}{(r^{*}+\epsilon)^{|n|}}.
      \end{align}
      where $C$ is a constant independent of $\delta$. If $|x|=r^{*}=\frac{r^{3}_{3}}{r^{2}_{2}}$,
      \begin{align*}
       \frac{2(r^{2|n|}_{2}-r^{2|n|}_{3})\gamma_{\delta}}
       {|n|r^{|n|-1}_{3}(4\gamma^{2}_{\delta}-\rho^{2|n|})}\frac{\widetilde{h}^{(n)}}{r^{|n|}}
       = \frac{2(1-\rho^{2|n|})\gamma_{\delta}}{(\rho^{|n|}-4\gamma^{2}_{\delta}\rho^{-|n|})}
       \frac{r^{|n|}_{2}\widetilde{h}^{(n)}}{|n|r_{3}^{|n|-1}}
      \end{align*}
      By Theorem 5.4 of \cite{Ammari2013a}, it follows that
      \begin{align*}
       \bigg| \frac{2(1-\rho^{2|n|})\gamma_{\delta}}{(\rho^{|n|}-4\gamma^{2}_{\delta}\rho^{-|n|})}\bigg| \le
       2\bigg( \frac{\delta}{4+\delta^{2}} \rho^{-|n|} + \frac{1}{\delta}\rho^{|n|}\bigg)^{-1}
      \end{align*}
      According to \eqref{thm3-tuidao}, we have
      \begin{align}\label{V-dacay}
        |V_{\delta}(x)| \le C\sum_{n\ne 0}\bigg( \frac{\delta}{4+\delta^{2}} \rho^{-|n|} + \frac{1}{\delta}\rho^{|n|}\bigg)^{-1}
        \bigg( \frac{r_{2}}{r^{*}+\epsilon}\bigg)^{|n|}
      \end{align}
      and hence
      \begin{align*}
        |V_{\delta}(x)| \to 0 \quad as \quad \delta \to 0.
      \end{align*}
      We now turn to the estimation of $T_{\delta}(x)$. From \eqref{S1D1D2}, $T_{\delta}(x)$ can be written as
      \begin{align*}
        T_{\delta}(x)= \mathcal{S}_{D_{1}}(\lambda_{1}I-\mathcal{K}^{*}_{D_{1}})^{-1}
    	\frac{\partial V_{\delta}}{\partial\nu_{1}}(x),
      \end{align*}
      $\mathcal{S}_{D_{1}}(\lambda_{1}I-\mathcal{K}^{*}_{D_{1}})^{-1}$ is a bounded linear operator and $D_{1}\cap B_{r^{*}}=\emptyset$.
      In order to estimate the normal derivative of $V_{\delta}$ on $D_{1}$, we first estimate the gradient of
      $V$ on $\mathbb{R}^{2} \backslash \overline{B_{r^{*}}}$. A direct computation of \eqref{Vdelta} yields
      \begin{align*}
       |\nabla V_{\delta}(x)| \le C\sum_{n\ne 0}\bigg( \frac{\delta}{4+\delta^{2}} \rho^{-|n|} + \frac{1}{\delta}\rho^{|n|}\bigg)^{-1}
       \frac{|n|}{r^{*}}\bigg( \frac{r_{2}}{r^{*}+\epsilon}\bigg)^{|n|}
      \end{align*}
      Thus, it follows that
      \begin{align*}
      |T_{\delta}(x)| \to 0 \quad as \quad \delta  \to 0.
      \end{align*}
     Since $V_{\delta}$ and $\nabla V_{\delta}$ is harmonic in $|x|>r_{3}$ and tends to $0$ as $|x| \to \infty$, applying the maximum principle together with the continuity of the operator $\mathcal{S}_{D_{1}}(\lambda_{1}I-\mathcal{K}^{*}_{D_{1}})^{-1}$, we
    obtain \eqref{thm-result}. This completes the proof.
    \end{proof}
    From \Thmref{thm5} and \eqref{totalpotential_decompose}, we have
    \[\lim_{\delta \to 0} \sup_{x \ge r^{*}} |u(x)| = |u_{D_{1}}(x)|,\]
     which shows that for small $\delta$, the total field $u$ can be approximated by $u_{D_{1}}$. Moreover, if $u_{D_{2},D_{3}}$ decays to zero sufficiently fast, then the iterative numerical reconstruction of $D_{1}$ can be carried out by simply solving $u_{D_{1}}$ at each step, rather than repeatedly solving the original equation for $u$ at each step, thereby significantly improving computational efficiency. Numerical evidence is provided in Section 5.

\section{Shape reconstruction via active cloaking}
\subsection{Reconstruction algorithm}

Two reconstruction strategies are considered for recovering the unknown penetrable sample $\partial D_1$ from data on $\partial B_R$. The reduced-domain method, neglects the probe entirely during the inversion process. The full-domain method, in contrast, solves the full probe-sample coupled system at each iteration.

The common framework for both methods is as follows. The boundary $\partial D_1$ is parameterized in star-shaped form as
\begin{align*}
	x(t) = r(t)(\cos t, \sin t), \quad 0 \le t \le 2\pi,
\end{align*}
with
\begin{align*}
	r(t) = \alpha_0 + \sum_{j=1}^{M} \left[ \alpha_j \cos(jt) + \alpha_{j+M} \sin(jt) \right],
\end{align*}
and the unknown coefficient vector is $\boldsymbol{\alpha} = [\alpha_0, \alpha_1, \ldots, \alpha_{2M}]^\top \in \mathbb{R}^{2M+1}$. We employ the Levenberg--Marquardt algorithm to minimize the residual between the measured and computed fields on $\partial B_R$. At each iteration, the update $\delta\boldsymbol{\alpha}$ solves
\begin{align*}
	\left( J^\top J + \lambda C \right) \delta\boldsymbol{\alpha} = -J^\top \big( u(\boldsymbol{\alpha}) - u^{\rm meas} \big),
\end{align*}
where $J$ is the Jacobian, $\lambda > 0$ is the regularization parameter, and $C$ is the diagonal matrix associated with the $H^1$-seminorm penalization of the Fourier coefficients. The linear system is solved by the conjugate gradient method, with the initial guess chosen as a circle of radius $r_0$. The two methods differ only in the forward solver: the reduced method excludes the probe, while the full-order method includes it. The corresponding algorithms are summarized in Algorithm~\ref{alg:reduced} and Algorithm~\ref{alg:full}.

\begin{algorithm}
	\caption{Reduced-domain method}
	\label{alg:reduced}
	\begin{algorithmic}
		\State{Input: Measured data $u^{\rm meas}$ on $\partial B_R$, initial radius $r_0$, regularization parameter $\lambda>0$, tolerance $\tau>0$, maximum iterations $K_{\max}$.}
		\State{Output: Reconstructed coefficient vector $\boldsymbol{\alpha}_{\rm red}$.}
		\State{Initialize $\boldsymbol{\alpha}^{(0)}=[r_0,0,\ldots,0]^\top$, $k=0$.}
		\While{$k<K_{\max}$ and $\|u(\boldsymbol{\alpha}^{(k)})-u^{\rm meas}\|>\tau$}
		\State{Solve \eqref{uD1} to obtain $u(\boldsymbol{\alpha}^{(k)})$ and the Jacobian $J(\boldsymbol{\alpha}^{(k)})$.}
		\State{Solve $(J^\top J+\lambda C)\delta\boldsymbol{\alpha}=-J^\top(u(\boldsymbol{\alpha}^{(k)})-u^{\rm meas})$ using CG.}
		\State{Update $\boldsymbol{\alpha}^{(k+1)}=\boldsymbol{\alpha}^{(k)}+\delta\boldsymbol{\alpha}$.}
		\State{Set $k\leftarrow k+1$.}
		\EndWhile
		\Return $\boldsymbol{\alpha}_{\rm red}=\boldsymbol{\alpha}^{(k)}$.
	\end{algorithmic}
\end{algorithm}

\begin{algorithm}
	\caption{Full-domain method}
	\label{alg:full}
	\begin{algorithmic}
		\State{Input: Measured data $u^{\rm meas}$ on $\partial B_R$, known probe geometry and material properties, initial radius $r_0$, regularization parameter $\lambda>0$, tolerance $\tau>0$, maximum iterations $K_{\max}$.}
		\State{Output: Reconstructed coefficient vector $\boldsymbol{\alpha}_{\rm full}$.}
		\State{Initialize $\boldsymbol{\alpha}^{(0)}=[r_0,0,\ldots,0]^\top$, $k=0$.}
		\While{$k<K_{\max}$ and $\|u(\boldsymbol{\alpha}^{(k)})-u^{\rm meas}\|>\tau$}
		\State{Solve \eqref{Integral-equations-matrix} to obtain $u(\boldsymbol{\alpha}^{(k)})$ and the Jacobian $J(\boldsymbol{\alpha}^{(k)})$.}
		\State{Solve $(J^\top J+\lambda C)\delta\boldsymbol{\alpha}=-J^\top(u(\boldsymbol{\alpha}^{(k)})-u^{\rm meas})$ using CG.}
		\State{Update $\boldsymbol{\alpha}^{(k+1)}=\boldsymbol{\alpha}^{(k)}+\delta\boldsymbol{\alpha}$.}
		\State{Set $k\leftarrow k+1$.}
		\EndWhile
		\Return $\boldsymbol{\alpha}_{\rm full}=\boldsymbol{\alpha}^{(k)}$.
	\end{algorithmic}
\end{algorithm}
\subsection{Numerical results and discussions}
In all numerical examples presented in this section, the measurement data are collected on $\partial B_R$ with $R=0.6$, and the initial guess for the reconstruction is chosen as a circle of radius $r_0=0.4$. The forward problem is solved using the Nystr\"om method, where each boundary is discretized into $N$ uniformly distributed collocation points. The reconstruction parameters are set to $K_{\max}=10$ and $M=7$.

We first consider the small-probe regime, where the reduced-domain method is applied to compare the reconstructions using a bare probe and a cloaked probe. We then turn to the large-probe regime, where both the reduced and full-domain methods are implemented and compared. Several probe configurations are examined, including single-layer, single multi-layer, and multiple multi-layer probes.
\subsubsection{Small-sized probe}

We first validate the reduced method in the small-probe regime. The unknown sample is set as a disk of radius $0.1$ centered at the origin. To avoid inverse crime, synthetic data are generated with $2N=64$ collocation points, while the inversion uses $N=32$. A single incident wave with direction $d=(\cos\theta_d,\sin\theta_d)$, $\theta_d=0$, is employed. The reduced-domain method is applied to both configurations below,

\begin{itemize}
	\item[(i)] \textbf{Single probe.} A bare probe (radius $0.005$, $\varepsilon_2=2$, $\varepsilon_3=\varepsilon_m=1$) and a single-layer cloaked probe ($D_2$: radius $0.005$, $\varepsilon_2=1$; $D_3$: radius $0.01$, $\varepsilon_3=-1+0.001i$) are placed at $(1,0)$, respectively.
	\item[(ii)] \textbf{Three probes.} Three bare probes and three single-layer cloaked probes placed at $(1,0)$, $(0,1)$ and $(0,-1)$.
\end{itemize}

\begin{figure}[htbp]
	\centering
	
	\subfloat
	{   \includegraphics[width=0.4\textwidth]{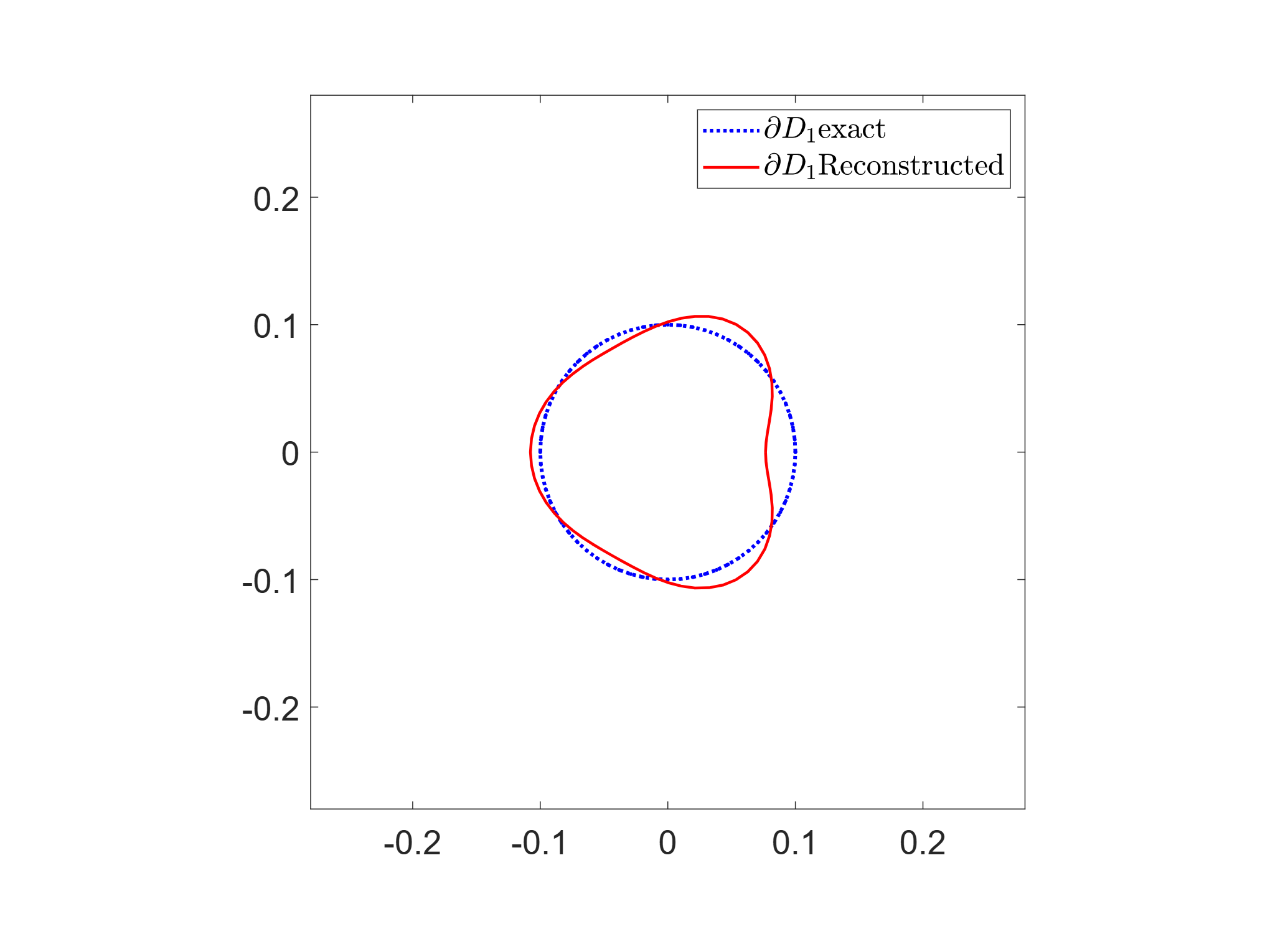}
	}\hspace{-1.5cm}
	{   \includegraphics[width=0.4\textwidth]{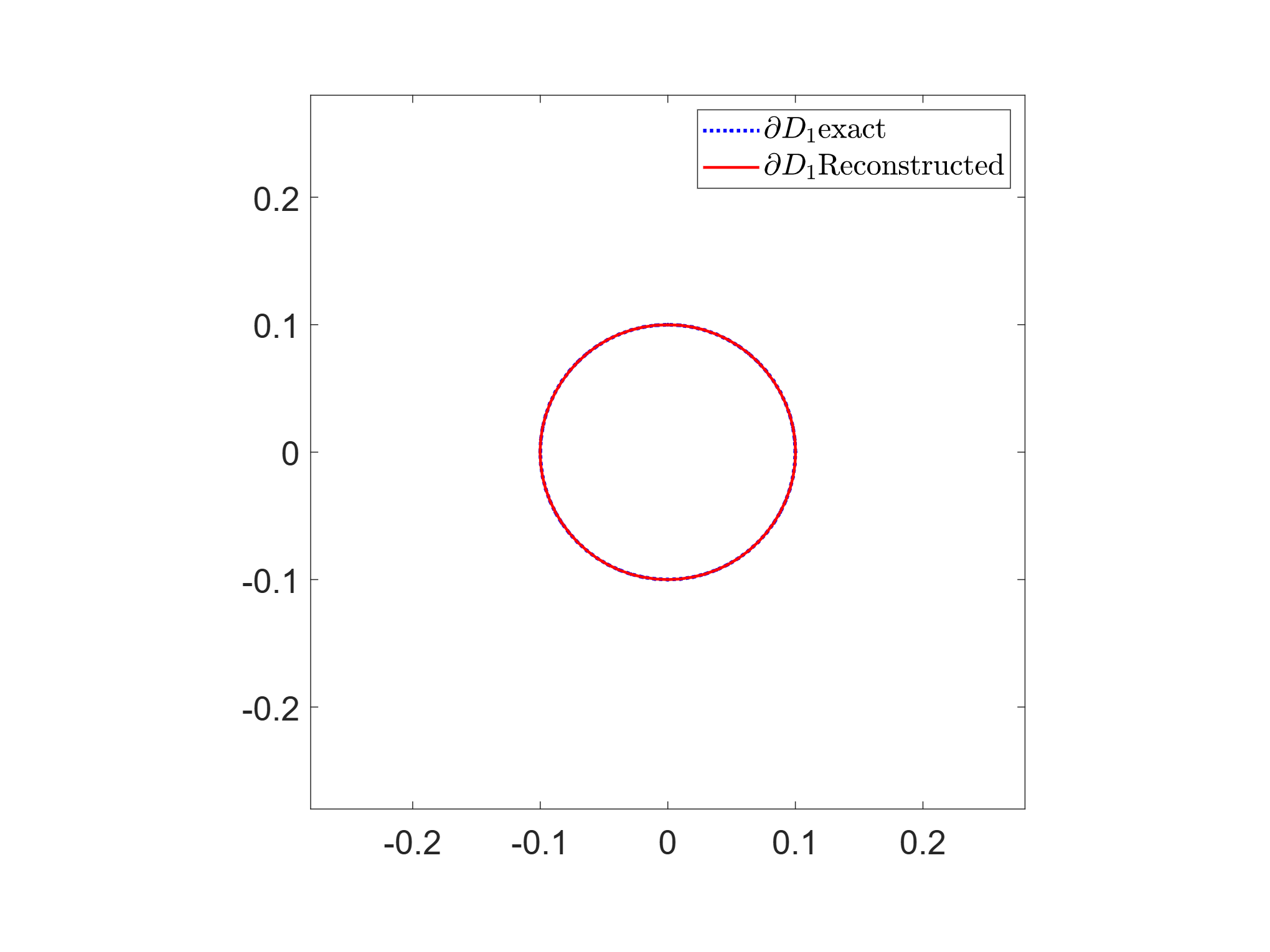}
	} \\ [-0.5cm]
	{   \includegraphics[width=0.4\textwidth]{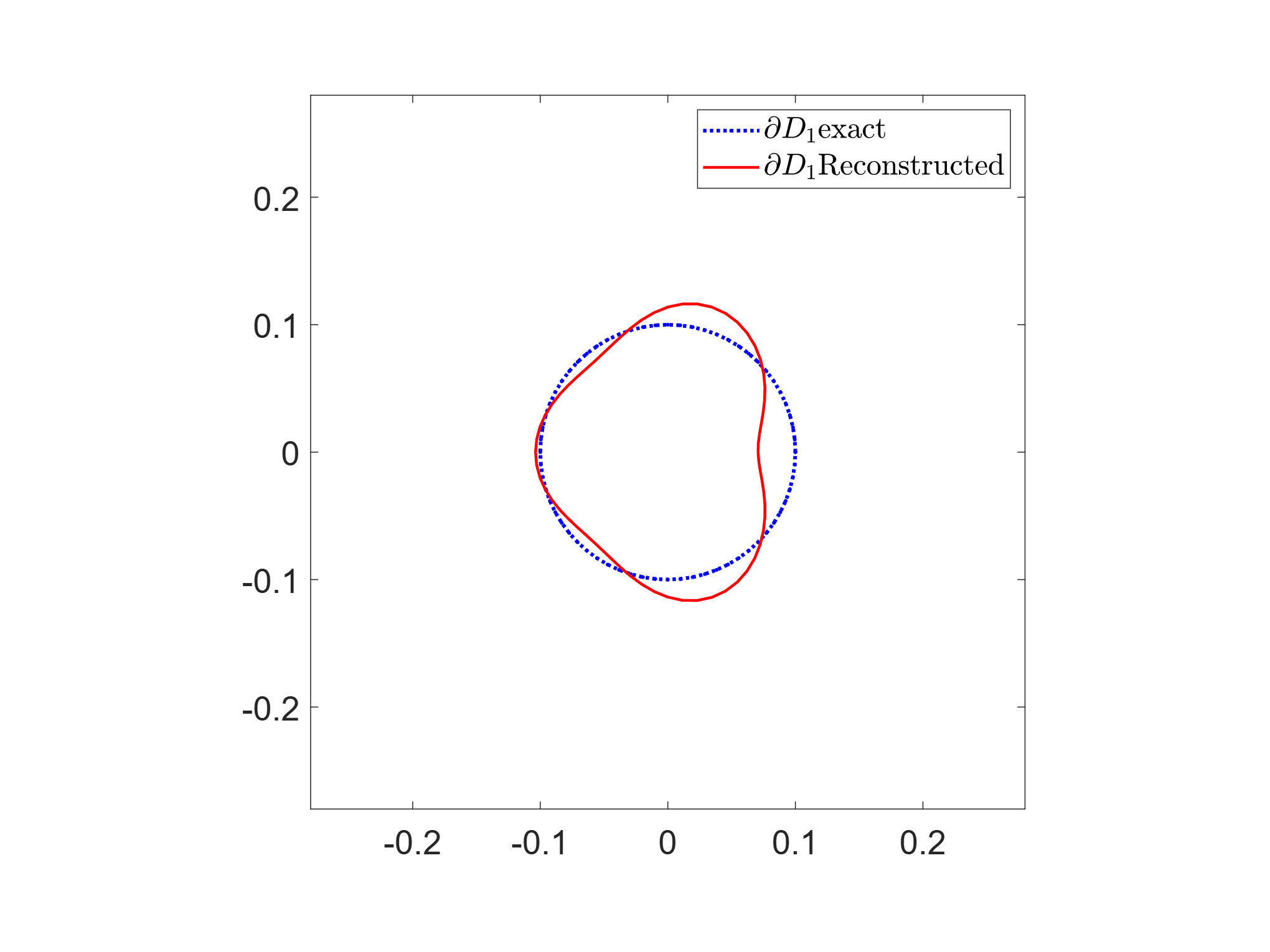}
	}\hspace{-1.5cm}
	{   \includegraphics[width=0.4\textwidth]{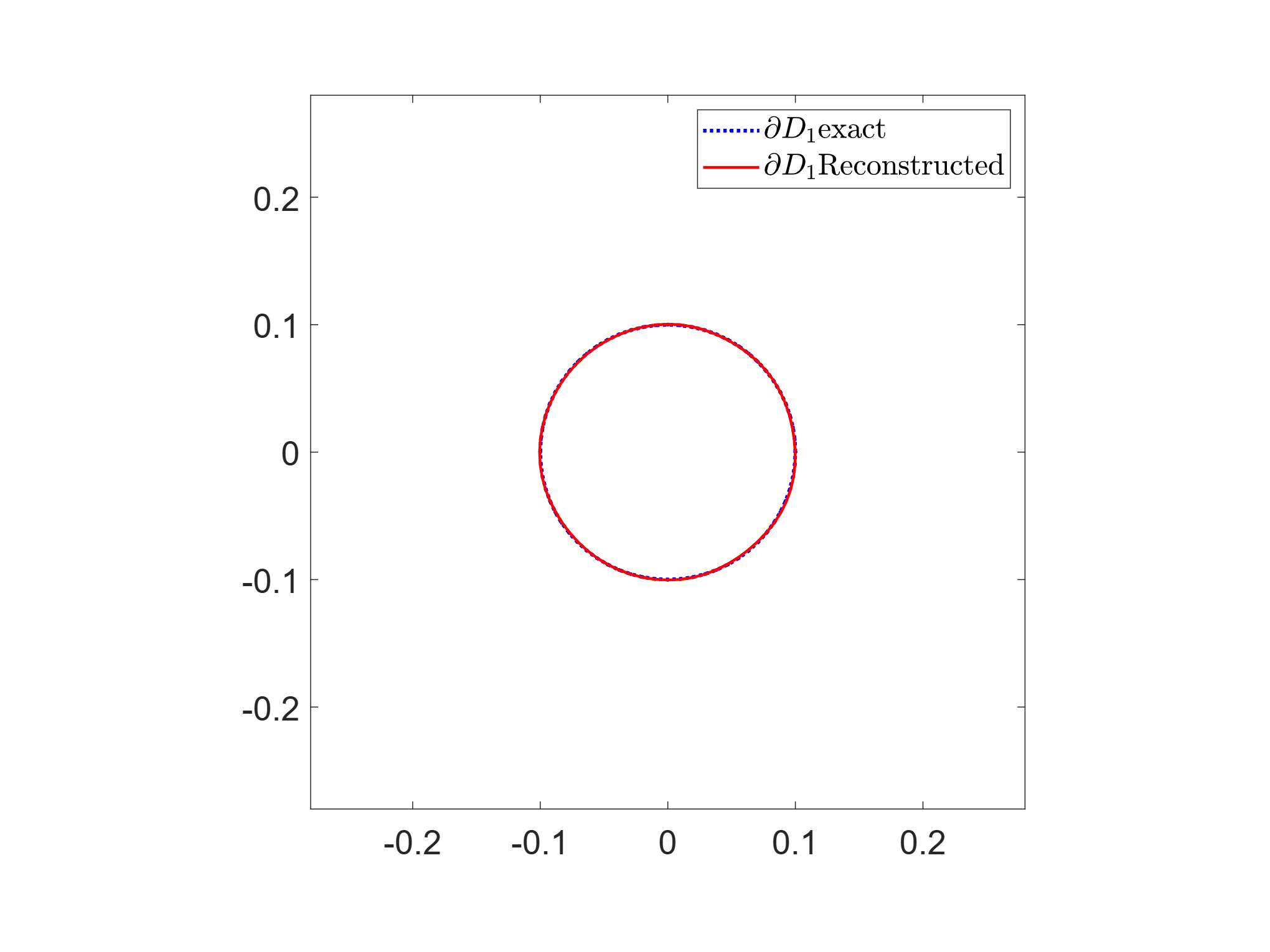}
	}
	\caption{Reconstruction results for the small-probe regime. Left: bare probe; right: cloaked probe; top: single probe; bottom: three probes.}
	\label{fig:small_probe}
\end{figure}
Figure~\ref{fig:small_probe} presents the reconstruction results for both configurations. The left column corresponds to the bare probe and the right column to the cloaked probe; the top row shows the single-probe results and the bottom row the multi-probe results. In all cases, the cloaked probe yields significantly more accurate reconstructions than the bare probe, demonstrating that the cloaking design effectively removes probe artifacts, allowing the reduced-domain method to achieve accurate reconstructions.
\subsubsection{Large-sized probe}
We now turn to the large-probe regime. In this setting, for an uncloaked probe, the reduced-domain method is no longer valid, and the full-domain method becomes necessary to ensure reconstruction accuracy.We will compare the reduced-remain and full-domain methods for various cloaked probe configurations to evaluate the cloaking performance in the large-probe regime.
\paragraph{Single-layer probe}
    We begin with the single-layer probe configuration. For ease of numerical result presentation, we let
	    \[\mathcal{N}_{1} = \norm{u_{D_{1}}-u^{i}}_{L^{2}(B_{R})}, \quad \mathcal{N}_{23} = \norm{u_{D_{2},D_{3}}}_{L^{2}(B_{R})}.\]
	    The incident wave is given by $d=(\cos (\pi/4), \sin (\pi/4))$. $D_{1}$ is a circle of radius 0.1 centered at the origin;
	    $D_{2}$ and $D_{3}$ are concentric disks centered at $(1,0)$ with radii $r_{2}=0.1$ and $r_{3} = 0.3$. The permittivities are chosen as $\epsilon_{1} = 1.2,\epsilon_{2} = \epsilon_{m} =1$ and $\epsilon_{3} = -s+ i\delta$ with $\delta=10^{-3}$. The measurement data are collected on
	    $B_{R}$ and $N=32$.
	
	    Figure~\ref{fig:J_s} shows the behavior of $\mathcal{J}(s)$ for $s\in[0,3]$. The functional attains its minimum at $s=1$, and blows up near $s=0.5$ and $s=2$. The blow-up occurs because $\lambda_{2}(s)$ approaches the eigenvalue $1/2 \rho$, as can be seen from \eqref{solution-potential} and \eqref{solution-S2S3}. Furthermore, we can directly compute that blow-up occurs at
	    \[s_{1} = \frac{1-\rho}{1+\rho}<1, \quad s_{2}=\frac{1+\rho}{1-\rho}>1.\]
	
     Figure~\ref{fig:N_vs_delta} shows $\mathcal{N}_{23}\to 0$ as $\delta\to 0$ with $s=1$, while $\mathcal{N}_1$ remains constant since $u_{D_1}$ is independent of $\delta$ and $s$. This confirms \Thmref{thm6}.
    \begin{figure}[htbp]
   \centering
 \subfloat[Variation of the minimized functional $\mathcal{J}$ with respect to $s$.]
  {   \includegraphics[width=0.4\textwidth]{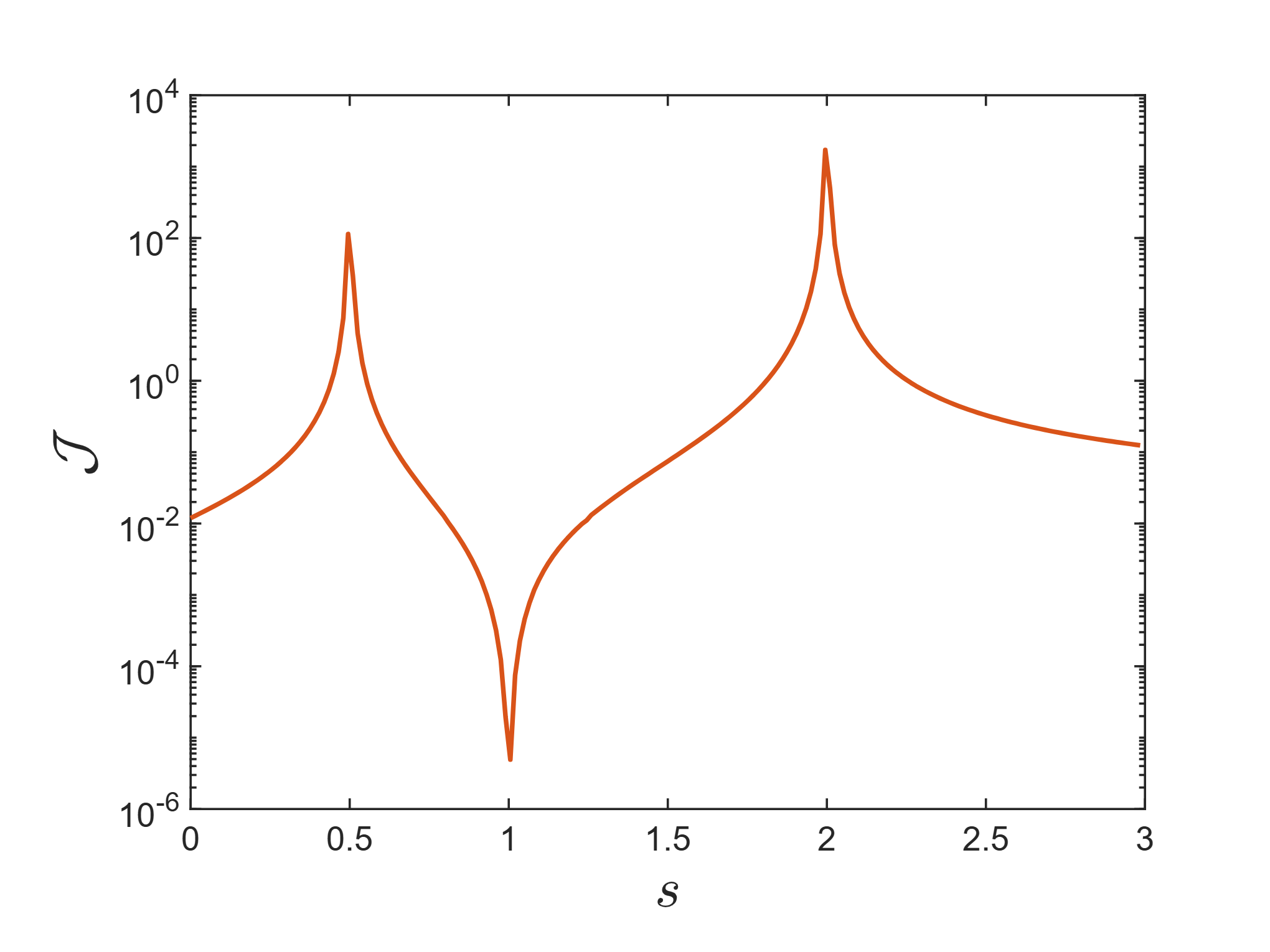} \label{fig:J_s}
 }
  \subfloat[Behavior of $\mathcal{N}_{1}$ and $\mathcal{N}_{23}$ as $\delta \to 0$.]
  {   \includegraphics[width=0.4\textwidth]{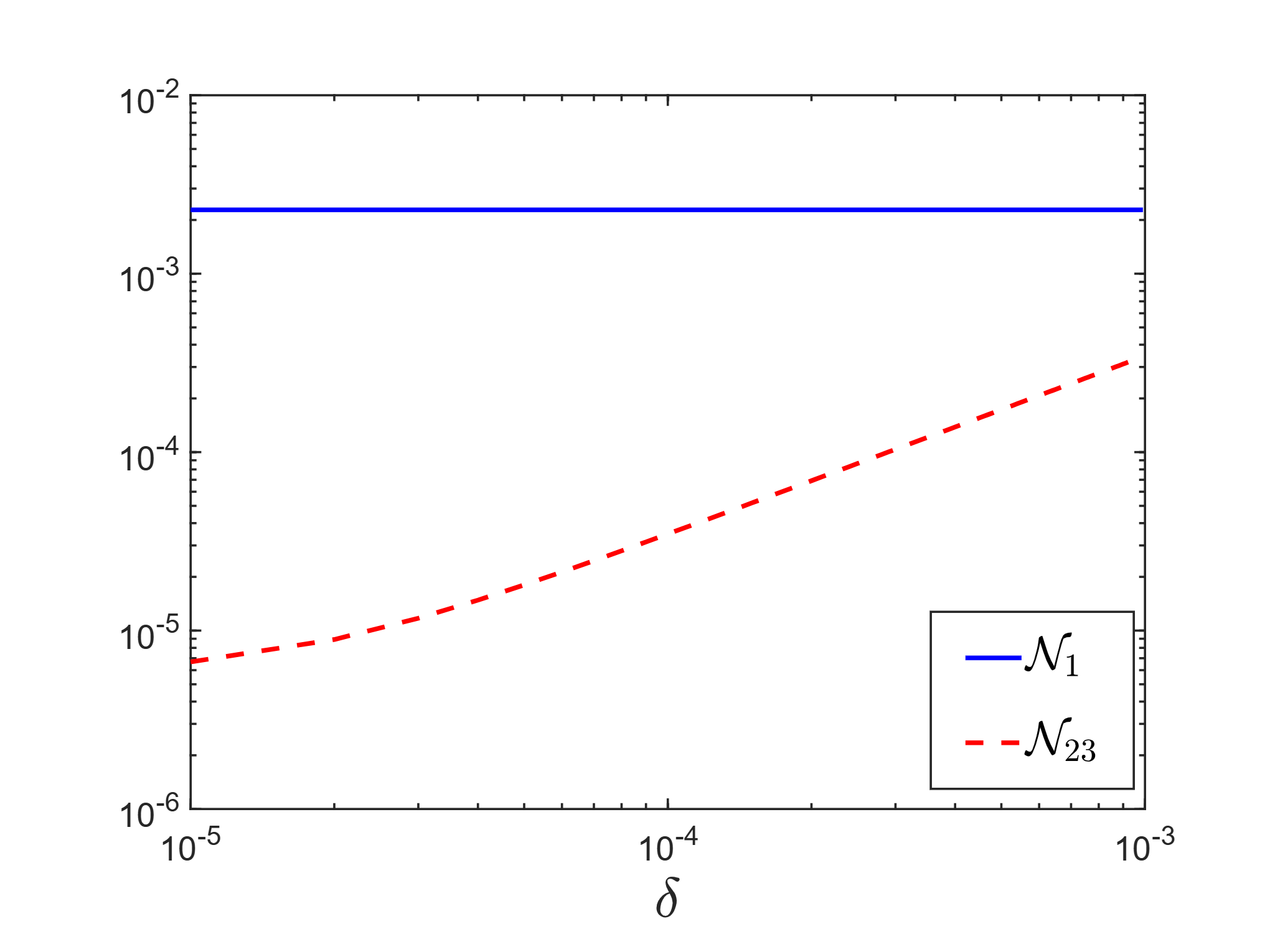} \label{fig:N_vs_delta}
  }
 \caption{(a) $\mathcal{J}(s)$: minimum at $s=1$, blow-up near $s=0.5$ and $s=2$. (b) For $s=1$, $\mathcal{N}_{23}\to 0$ as $\delta\to 0$, while $\mathcal{N}_1$ remains constant.}
 \label{fig:1}
\end{figure}

    With the above results, we are now ready to apply the reduced-domain method for the inversion. Keeping the same parameter settings as in the previous example, we consider three different shapes to be reconstructed: a circle of radius 0.1
     centered at the origin, an ellipse
     \[ x(t) = 0.2\cos{t} + 0.1\sin{t},\]
     and a kite-shaped curve
     \[x(t) = -0.05 + 0.1\cos{t} + 0.05\cos{2t} + 0.15\sin{t}.\]
     for $0\le t \le 2\pi$. Both the reduced-domain and full-domain methods are employed for the reconstruction. For the circle, a single incident direction $\theta_d=0$ is used; for the ellipse and kite-shaped curve, four incident directions $\theta_d=0,\pi/4,\pi/2,3\pi/4$ are employed.
      \begin{figure}[htb]
   \centering

 \subfloat
 {   \includegraphics[width=0.37\textwidth]{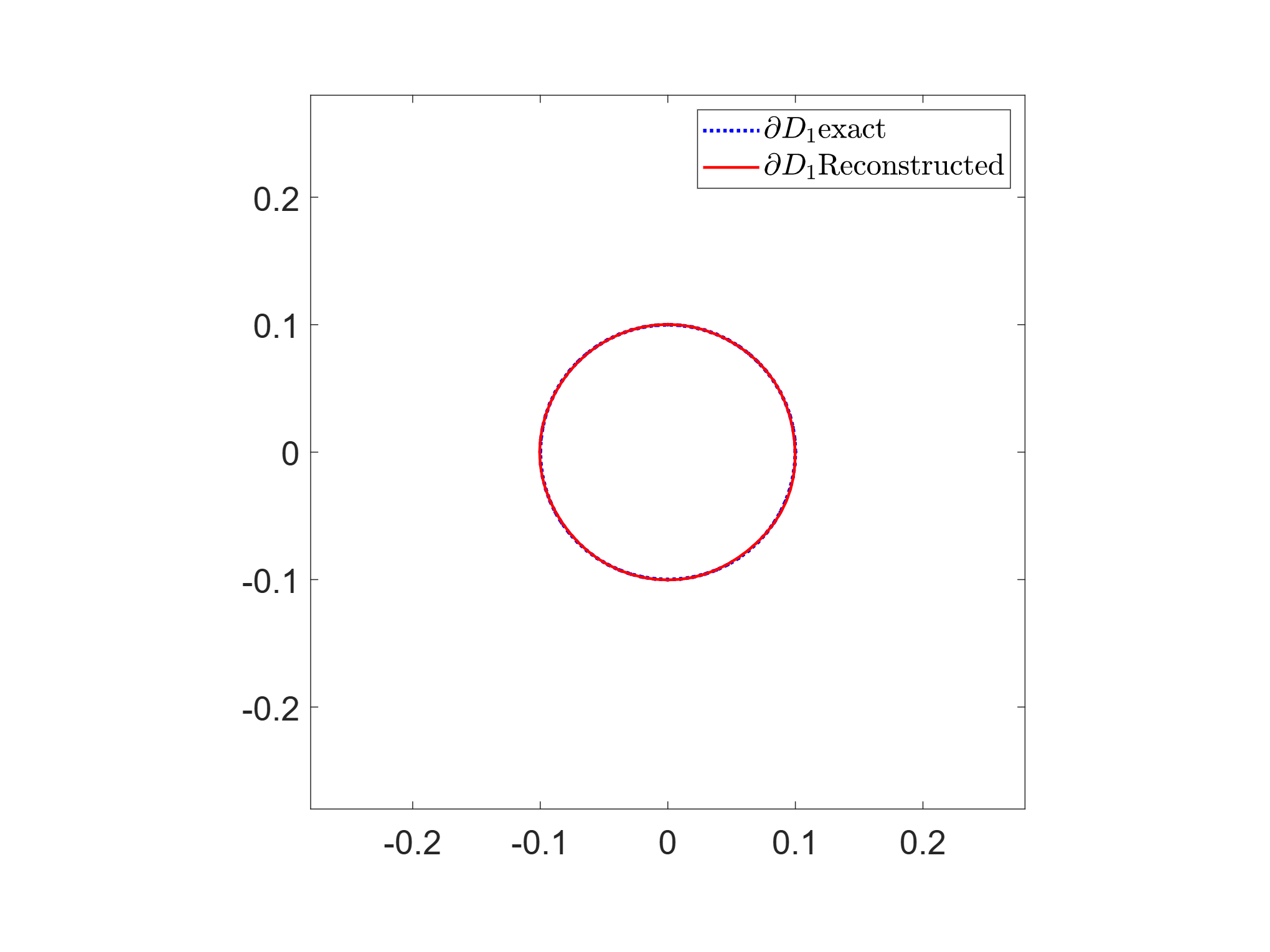}
 }\hspace{-1.4cm}
  {   \includegraphics[width=0.37\textwidth]{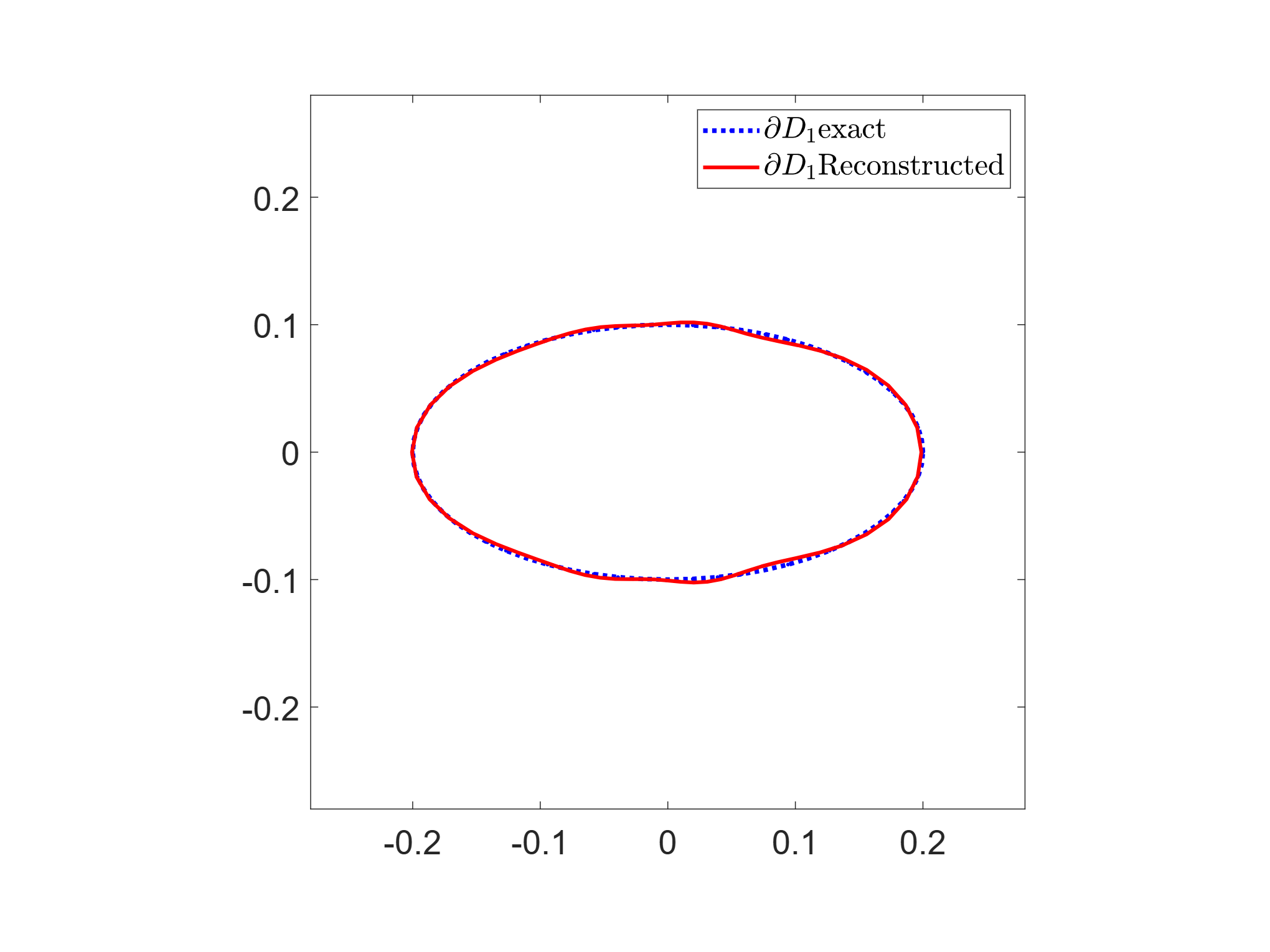}
 }\hspace{-1.4cm}
  {   \includegraphics[width=0.37\textwidth]{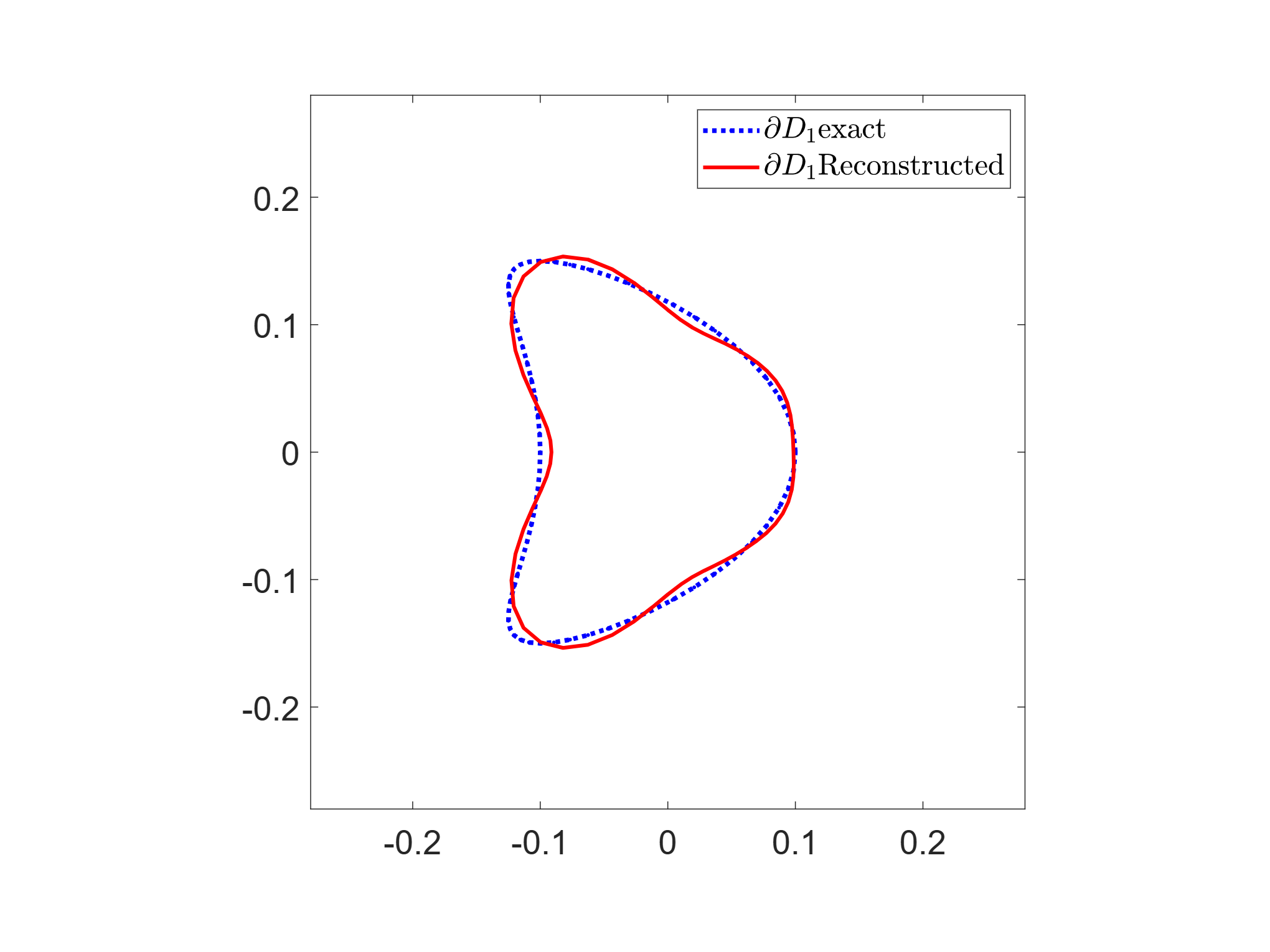}
  }  \\[-0.5cm]
 {   \includegraphics[width=0.37\textwidth]{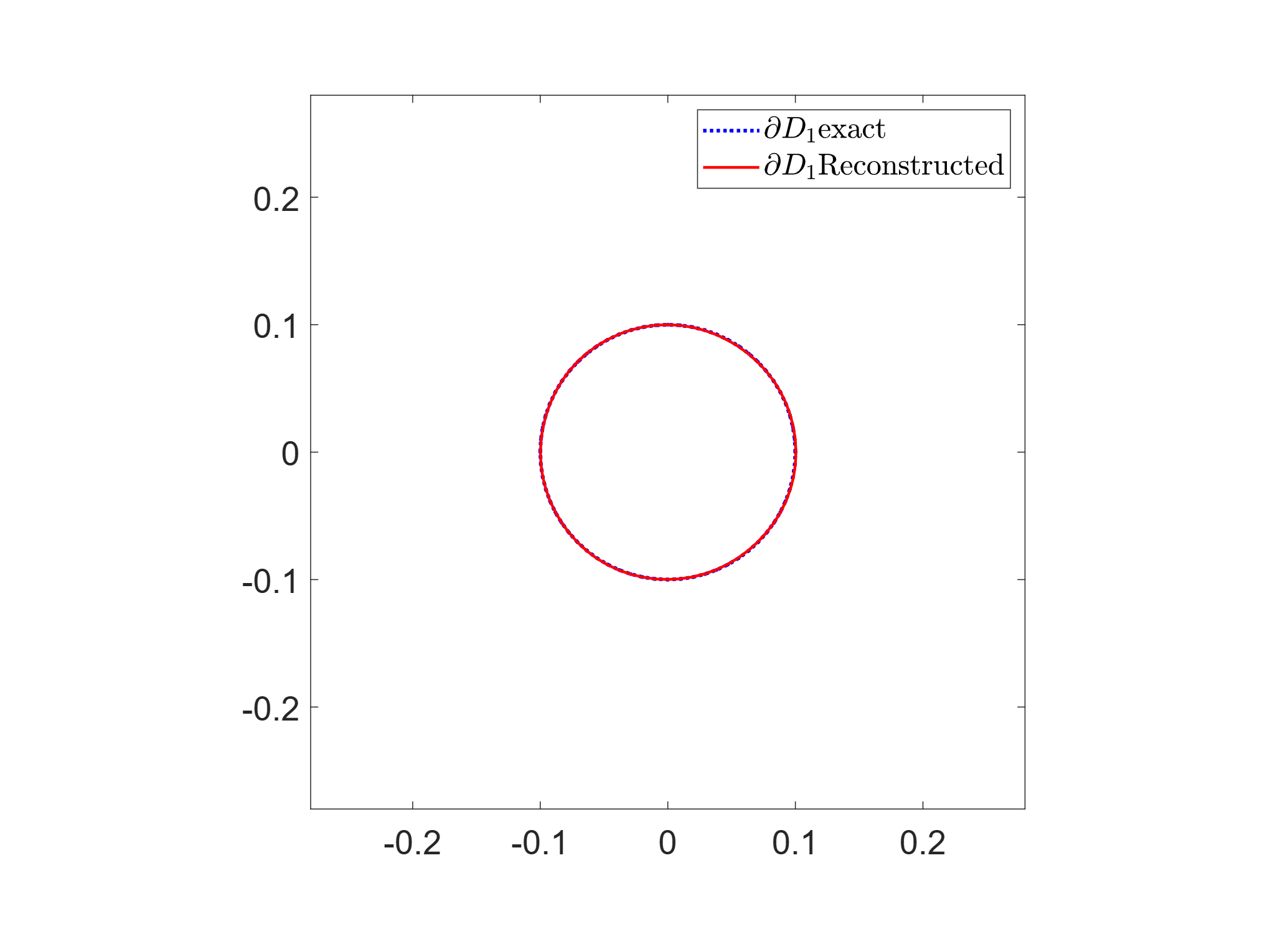}
 }\hspace{-1.4cm}
  {   \includegraphics[width=0.37\textwidth]{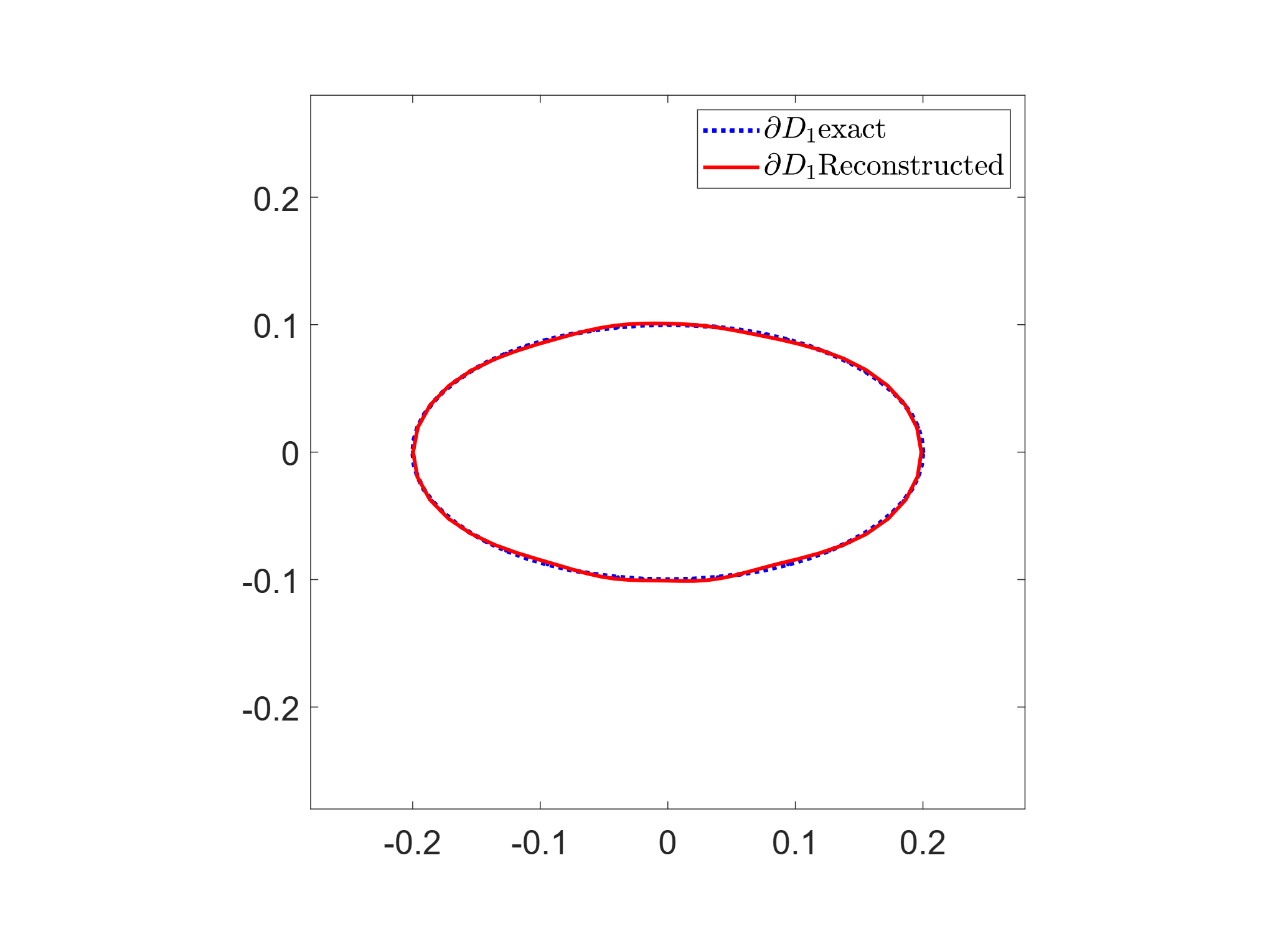}
 }\hspace{-1.4cm}
  {   \includegraphics[width=0.37\textwidth]{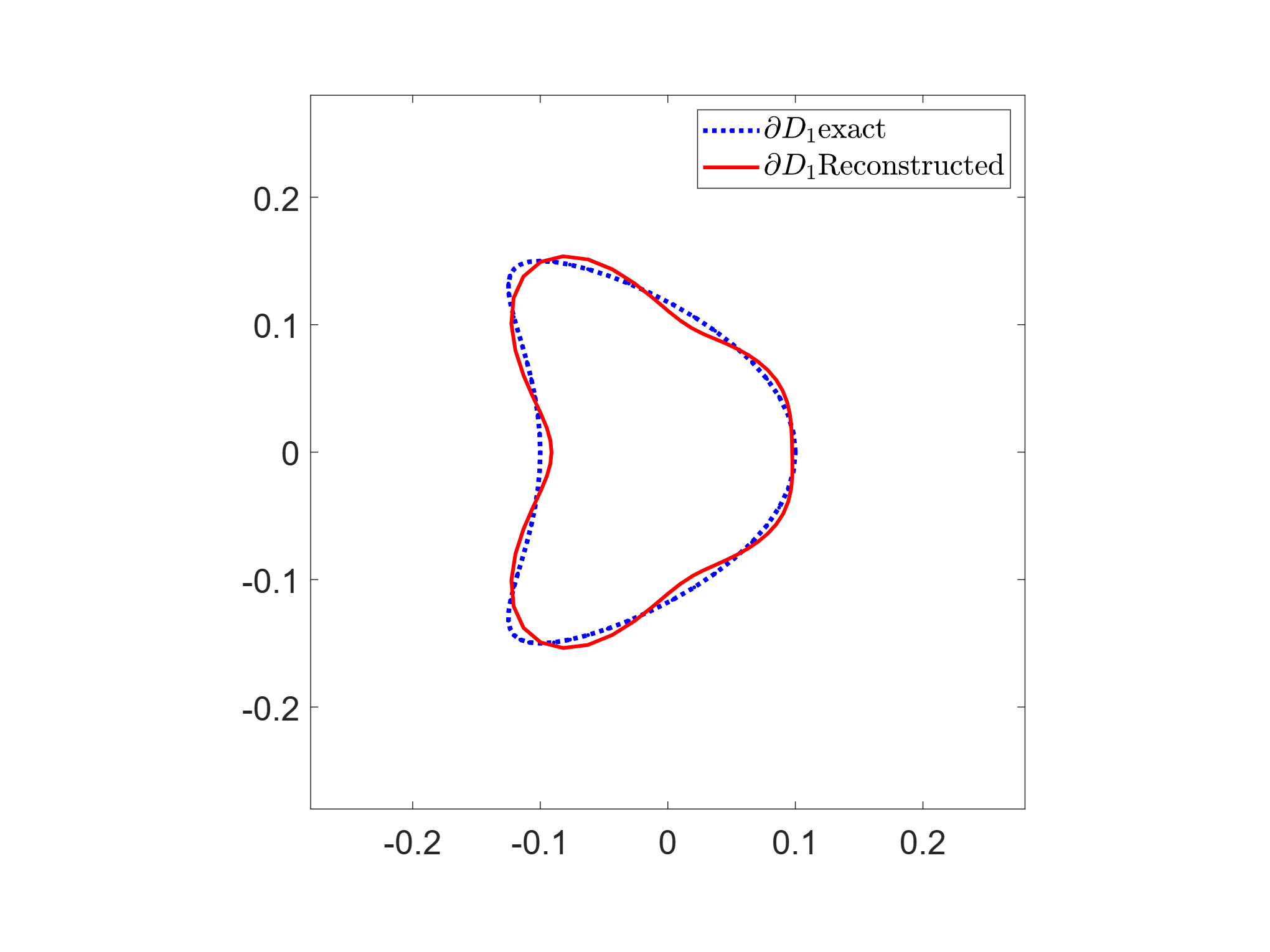}
  }
 \caption{Reconstruction results for $\partial D_{1}$. First row: full-domain method.
  Second row: reduced-domain method.}
 \label{fig:2}
\end{figure}

 Our numerical results are shown in Figure~\ref{fig:2}. The first row presents the results obtained by the full-domain method, while the second row presents the results obtained by the reduced-domain method.
  In Figure~\ref{fig:2}, the exact $\partial D_{1}$ is represented by a blue dotted line, and the reconstruction result is shown as a red curve. The two methods achieve comparable reconstruction accuracy, while the reduced-domain method is significantly more efficient.
  This efficiency gain arises because the reduced-domain method solves for a single density ($O(N\times N)$) per forward solve, whereas the full-domain method solves for three densities ($O(3N\times 3N)$). This advantage is further amplified over multiple iterations. A detailed timing comparison for multi-layer probes is given in the next section.(see Table~\ref{tab:efficiency1}).

\paragraph{Multi-layer probe}
\begin{figure}[htb]
\centering
\begin{tikzpicture}[scale=0.8]

\fill[fill=red!20] (0,0) circle (0.45);

\fill[blue!30, even odd rule] (0,0) circle (0.8) (0,0) circle (0.45);
\draw[black, thin] (0,0) circle (0.8);

\fill[gray!10, even odd rule] (0,0) circle (1.15) (0,0) circle (0.8);
\draw[black, thin] (0,0) circle (1.15);

\fill[blue!30, even odd rule] (0,0) circle (1.5) (0,0) circle (1.15);
\draw[black, thin] (0,0) circle (1.5);

\fill[gray!10, even odd rule] (0,0) circle (1.85) (0,0) circle (1.5);
\draw[black, thin] (0,0) circle (1.85);

\fill[blue!30, even odd rule] (0,0) circle (2.2) (0,0) circle (1.85);
\draw[black, thin] (0,0) circle (1.5);
\draw[black, thick] (0,0) circle (0.45);
\draw[black, thick] (0,0) circle (0.8);
\draw[black, thick] (0,0) circle (1.15);
\draw[black, thick] (0,0) circle (1.5);
\draw[black, thick] (0,0) circle (1.85);
\draw[black, thick] (0,0) circle (2.2);
\node at (-4.5, 0) {Air (background)};

\begin{scope}[shift={(3, 0.5)}]
    \fill[red!20] (0,0) circle (0.18);
    \draw[black, thin] (0,0) circle (0.18);
    \node[right] at (0.25, 0) {$D_2$ (core)};

    \fill[blue!30] (0,-0.6) circle (0.18);
    \draw[black, thin] (0,-0.6) circle (0.18);
    \node[right] at (0.25, -0.6) {Shell (coating, $\varepsilon=-1+i\delta$)};

    \fill[gray!10] (0,-1.2) circle (0.18);
    \draw[black, thin] (0,-1.2) circle (0.18);
    \node[right] at (0.25, -1.2) {Shell (air, $\varepsilon=1$)};
\end{scope}
\end{tikzpicture}
\caption{Schematic of the multi-layer core-shell probe}
\label{fig:multi_layer_geometry}
\end{figure}
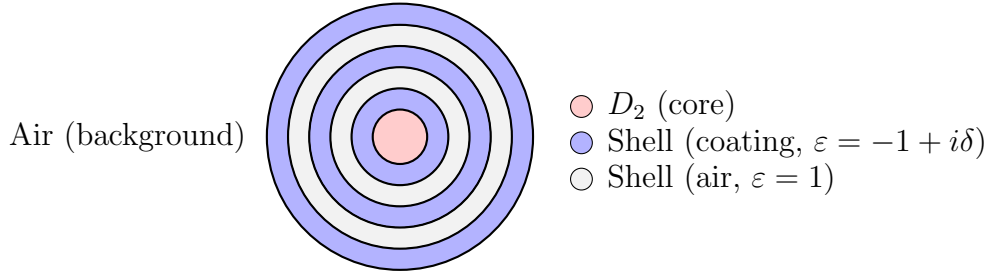
We now extend the numerical results to the case of multi-layer probes. The geometry of the multi-layer probe is illustrated in Figure~\ref{fig:multi_layer_geometry}. The probe consists of alternating coating and air shells. An $L$-layer probe has $L$ coating shells and $L-1$ air shells. In this section, the center of the multi-layer probe is located at $(2,0)$. The core $D_{2}$ has radius $0.1$ and permittivity $\varepsilon_{2} = 1$. Each shell (both coating and air) has thickness $0.1$.

To verify whether the multi-layer probe satisfies \Thmref{thm6}, we define $\mathcal{N}_{2} = \norm{u-u_{D_{1}}}_{L^{2}(B_{R})}$ for the multi-layer case. The parameters for $D_{1}$ (radius $0.1$, $\varepsilon_{1} = 1.2$) and the measurement setup ($R = 0.6$, $N = 32$) are the same as in the single-layer case. The results are computed in the limit $\delta \to 0$ and are presented in Figure~\ref{fig:3}. As $\delta \to 0$,
$\mathcal{N}_{2}$ approaches zero for both $L=3$ and $L=5$, showing that the multi-layer probe inherits the property established in \Thmref{thm6} for the single-layer case. Hence, the reduced-domain method remains applicable for reconstructing $\partial D_{1}$ in the presence of multi-layer probe.
\begin{figure}[htbp]
   \centering
   \subfloat[$L=3$]
    {   \includegraphics[width=0.4\textwidth]{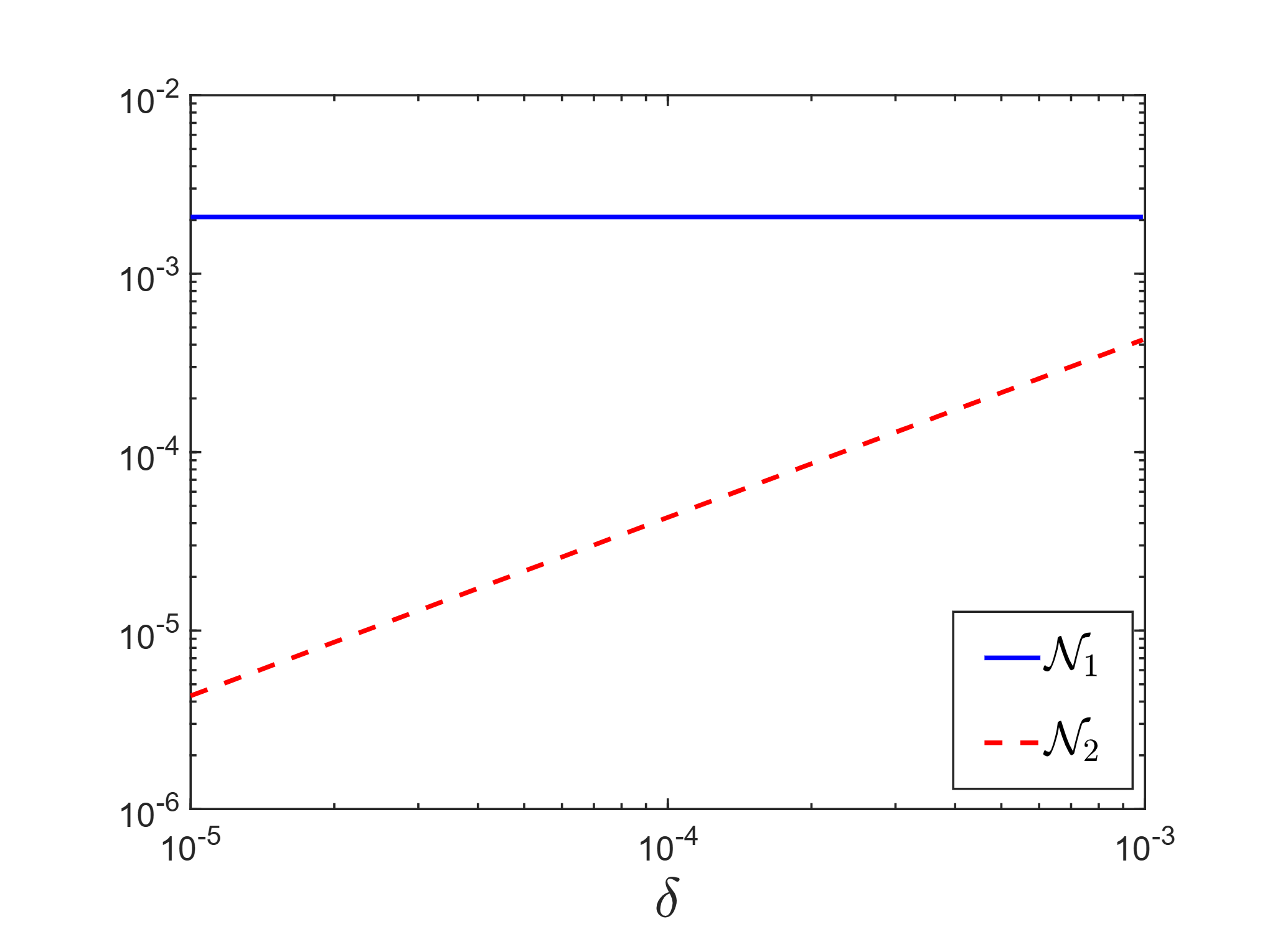}
     }
     \subfloat[$L=5$]
     {   \includegraphics[width=0.4\textwidth]{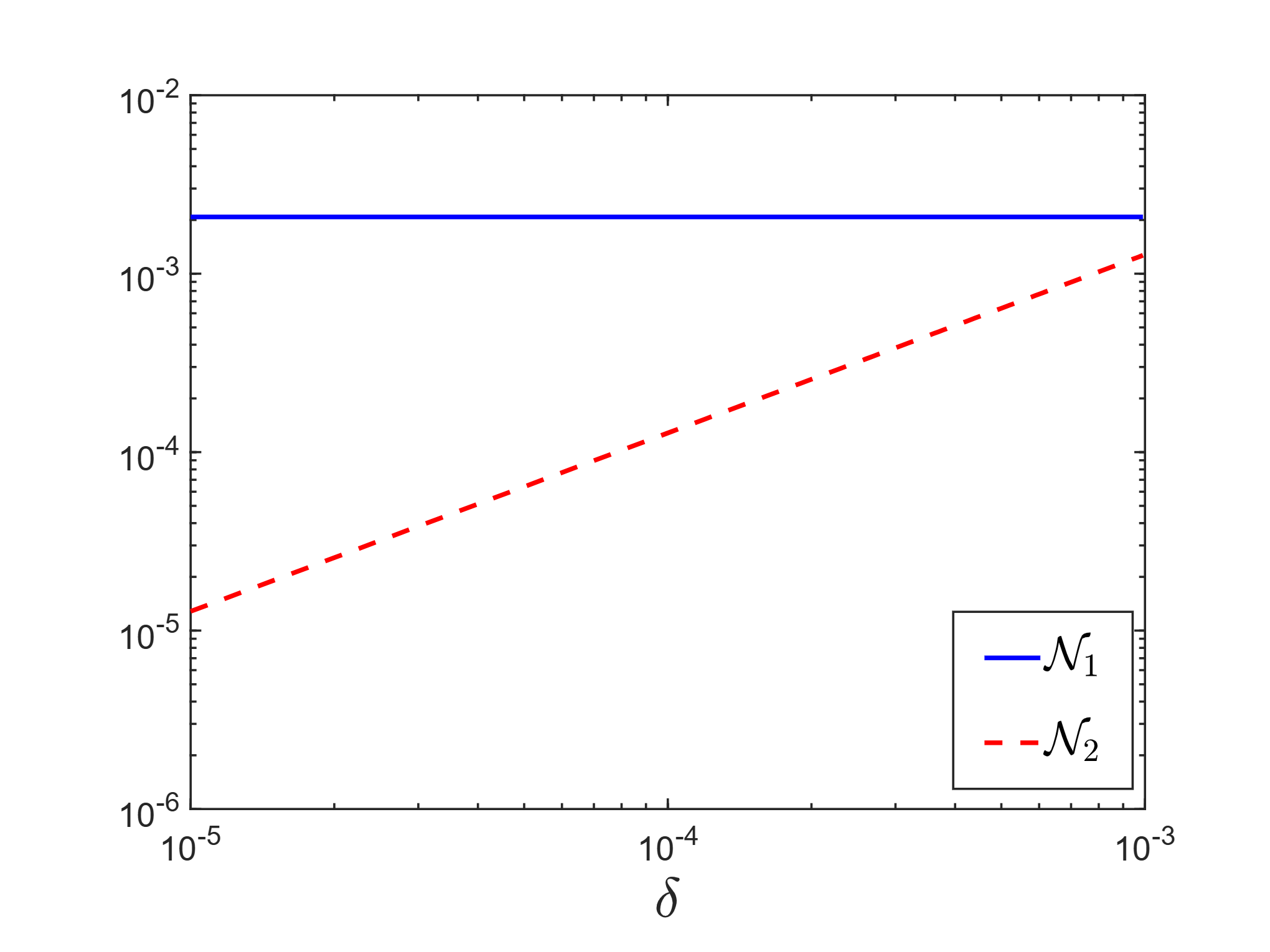}
     }
 \caption{Behavior of $\mathcal{N}_{1}$ and $\mathcal{N}_{2}$ as $\delta \to 0$ for an $L$-layer probe.}
 \label{fig:3}
\end{figure}

We now reconstruct $\partial D_1$ for multi-layer probes using the reduced-domain method. With $\delta=0.001$, we consider circle and kite-shaped targets for $L=3$ and $L=5$, using the same parameters as in Figure~\ref{fig:2}. The results, shown in Figure~\ref{fig:4}, demonstrate that the reduced-domain method achieves high accuracy and is significantly more efficient than the full-domain method.
\begin{figure}[htb]
   \centering

 \subfloat
 {   \includegraphics[width=0.4\textwidth]{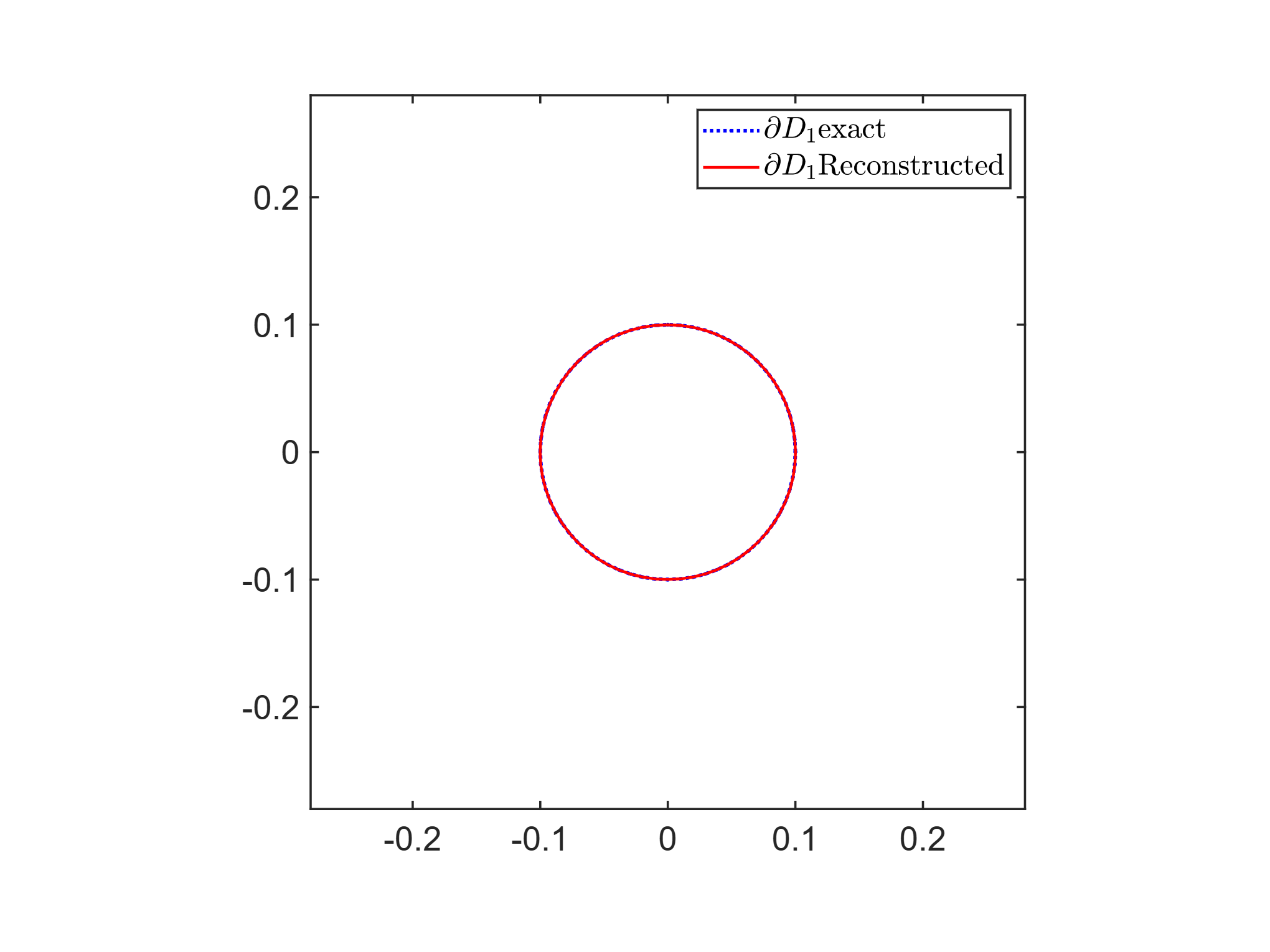}
 }\hspace{-1.5cm}
  {   \includegraphics[width=0.4\textwidth]{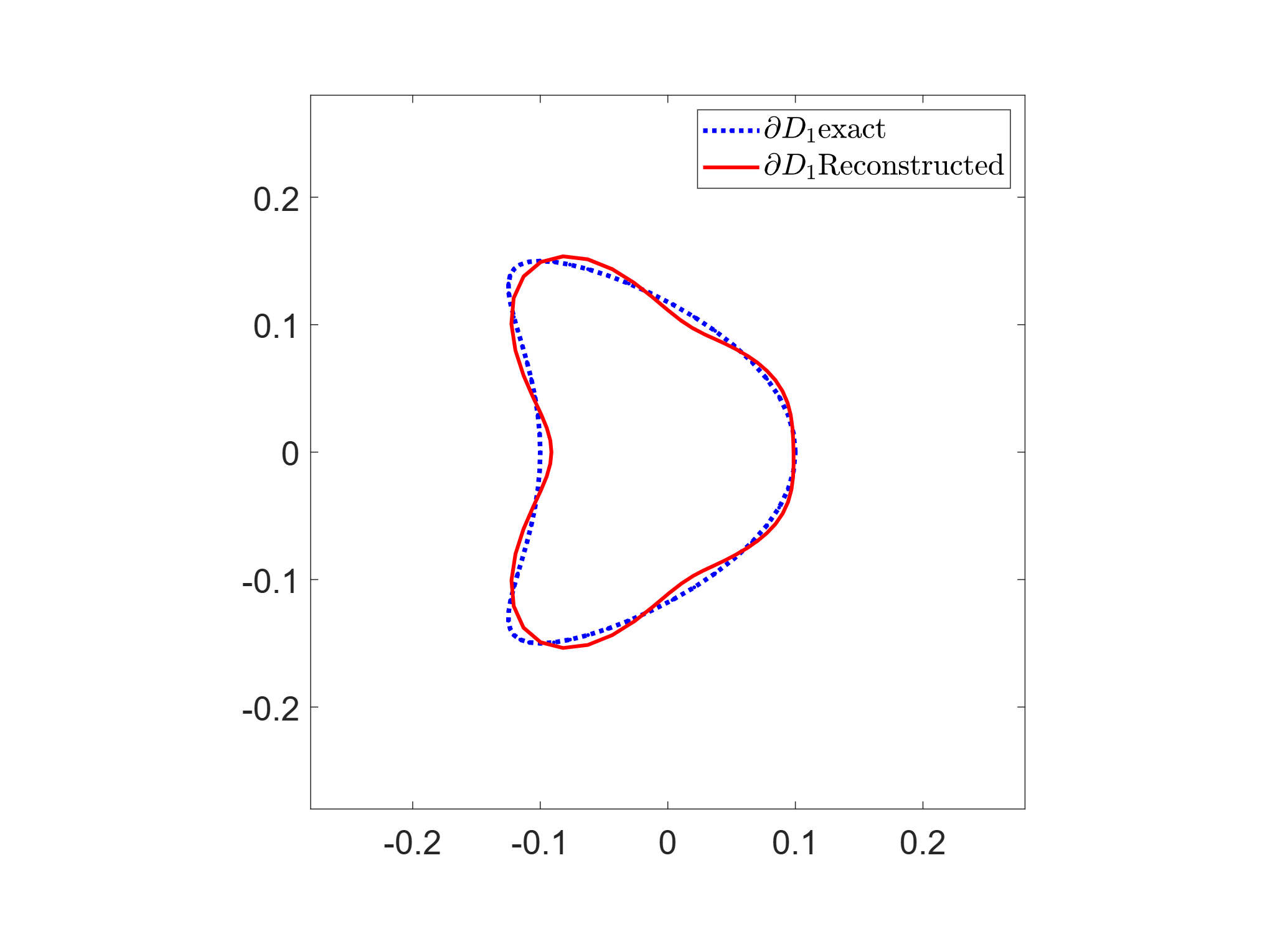}
  } \\ [-0.5cm]
 {   \includegraphics[width=0.4\textwidth]{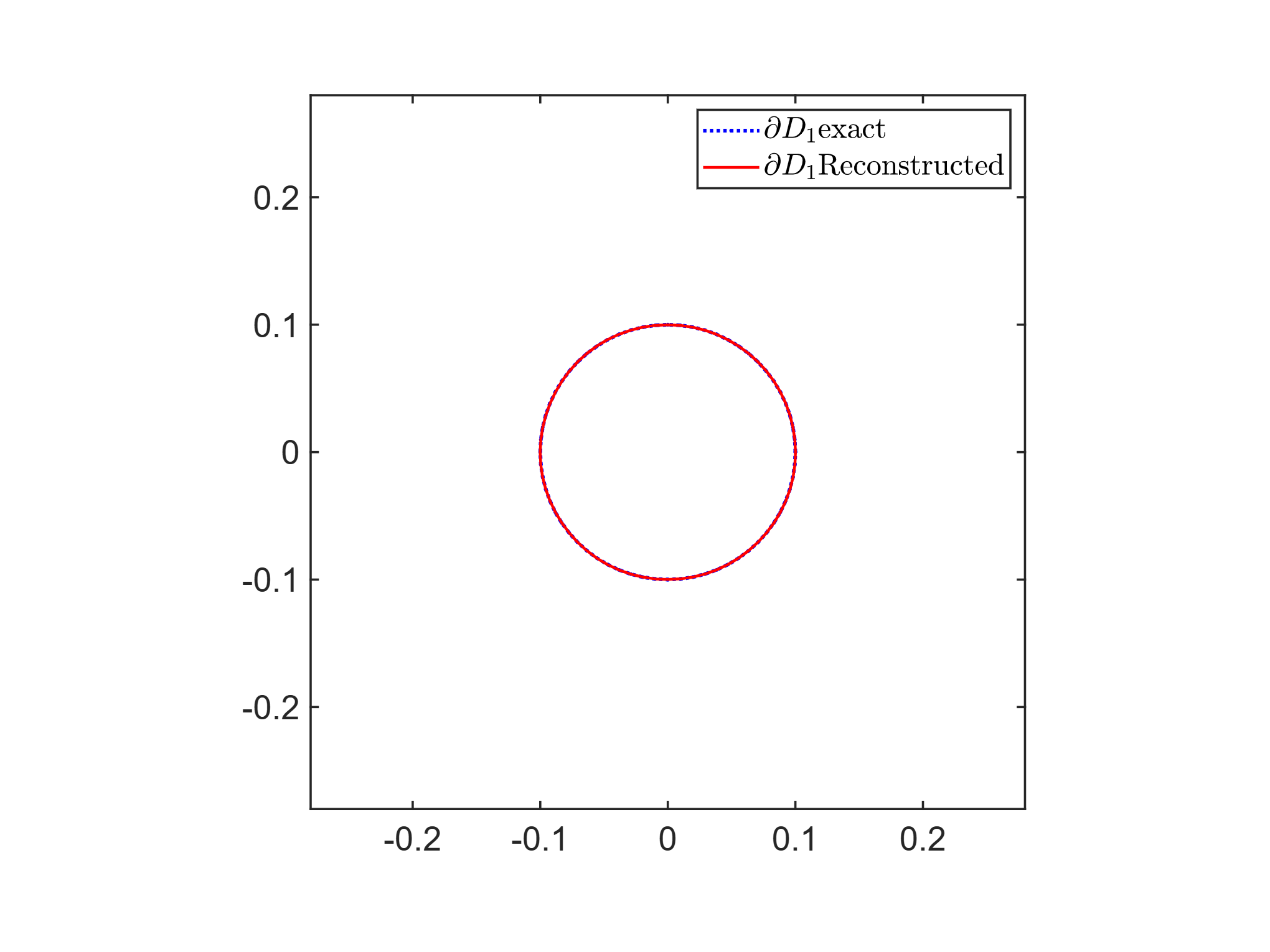}
 }\hspace{-1.5cm}
  {   \includegraphics[width=0.4\textwidth]{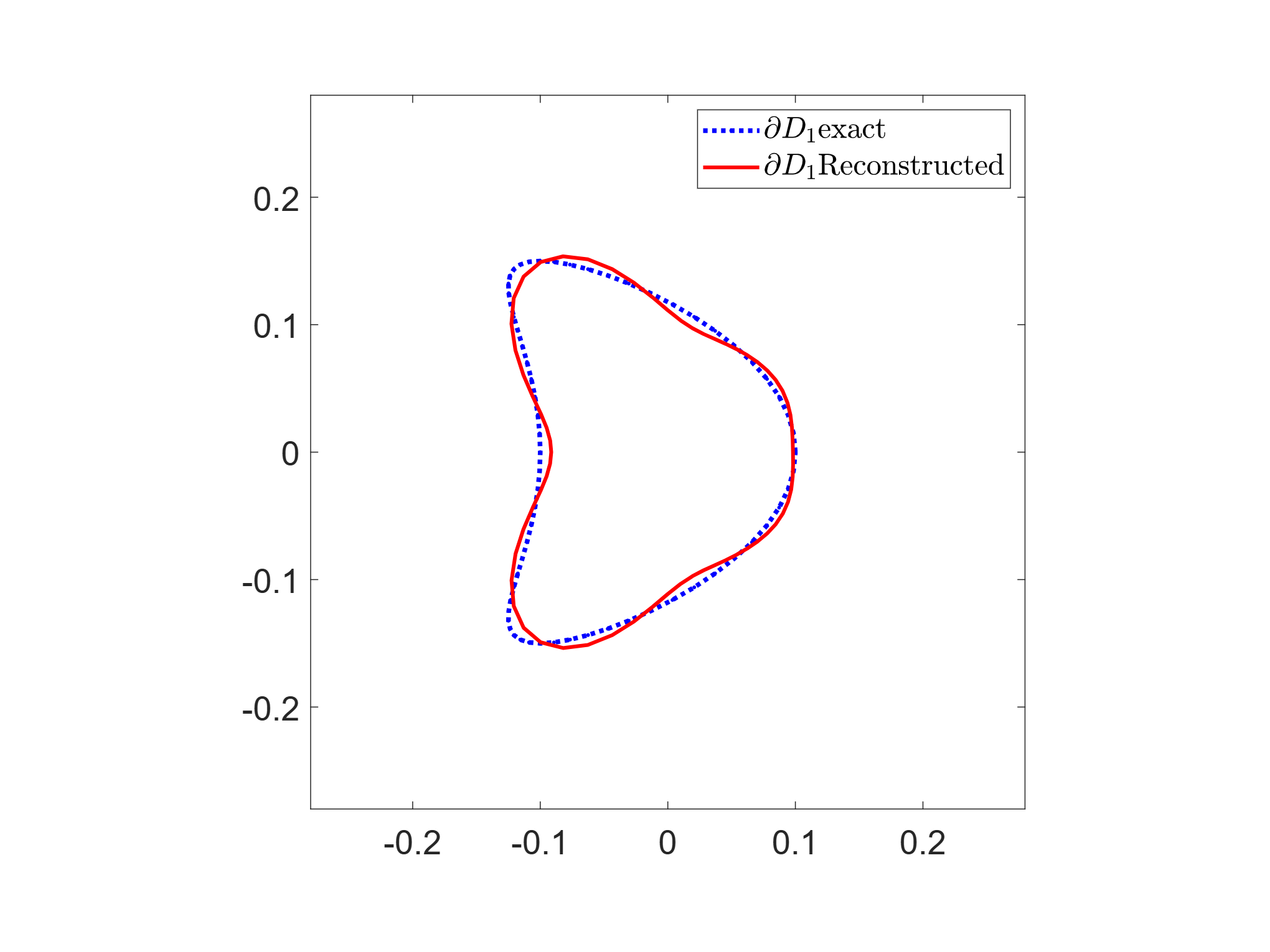}
  }
 \caption{Reconstruction results of the reduced-domain method for $L$-layer
 probe. Top row: $L=3$ (left: circle, right: kite). Bottom row: $L=5$ (left: circle, right: kite).}
 \label{fig:4}
\end{figure}

To quantitatively compare the computational efficiency of the two methods, we set $\delta = 0.001$ and $N = 128$. We consider $L = 1, 3, 5$ and three target shapes: circle, ellipse, and kite. For each configuration, we reconstruct $D_1$ using both the reduced-domain and full-domain methods, and record their computational times as $t_{\text{re}}$ and $t_{\text{fu}}$, respectively. Each experiment is repeated five times, and the average computation time is reported. The detailed efficiency comparison is presented in Table~\ref{tab:efficiency1}.

From Table~\ref{tab:efficiency1}, the reduced-domain method is significantly faster than the full-domain method. Note that the circle uses one set of measurements, while the ellipse and kite use four, which contributes to the longer computation time for the latter. The speedup factor of the reduced-domain method is about 4 for $L=1$, 20--23 for $L=3$, and 54--62 for $L=5$. Importantly, as $L$ increases, $t_{\text{re}}$ remains nearly constant while $t_{\text{fu}}$ grows rapidly, indicating that the advantage of the reduced-domain method grows as the number of probe layers increases.

This efficiency gain comes from the different linear system sizes in each forward solve. The reduced-domain method solves for a single density on $D_1$ ($O(N\times N)$), while the full-domain method solves for $2L+1$ densities ($O((2L+1)N\times(2L+1)N)$). This cost difference is amplified over multiple iterations, yielding the speedups reported in the table.
\begin{table}[htb]
\centering
\setlength{\tabcolsep}{18pt}
\caption{Computational time comparison between the reduced-domain and full-domain methods.}
\begin{tabular}{ccccc}
\toprule
Layers $L$ & Shape & $t_{\text{re}}$ (s) & $t_{\text{fu}}$ (s) & Speedup \\
\midrule
\multirow{3}{*}{1} & Circle & 0.6231 & 2.6481 & 4.25 \\
                   & Ellipse & 2.1319 & 9.8046  & 4.60 \\
                   & Kite & 2.1272 & 9.8169 & 4.61\\
\hline\hline
\multirow{3}{*}{3} & Circle & 0.6199 & 12.4798 & 20.13 \\
                   & Ellipse & 2.1690 & 47.4833 & 21.89 \\
                   & Kite & 2.1048 & 47.8806 & 22.75 \\
\hline\hline
\multirow{3}{*}{5} & Circle & 0.6105 & 33.0766 & 54.18 \\
                   & Ellipse & 2.1498 & 133.3211 & 62.02 \\
                   & Kite & 2.1493 & 133.1206 &  61.94\\
\bottomrule
\end{tabular}
\label{tab:efficiency1}
\end{table}

\paragraph{Multiple Probes}
\begin{figure}[htbp]
\centering
\begin{tikzpicture}[scale=1]

\begin{scope}[shift={(0,0)}]
    \begin{scope}[shift={(-1.0,0)}]
        \draw[fill=green!40, draw=black, thick, scale=0.35]
            plot[domain=0:360, smooth, samples=60]
            ({1.0*cos(\x) + 0.2*cos(2*\x)},
             {0.8*sin(\x) + 0.1*sin(2*\x)});
        \node[font=\small] at (0,0) {$D_1$};
    \end{scope}

    \foreach \y in {-1.0, 0, 1.0} {
        \fill[red!20] (1.6,\y) circle (0.18);
        \fill[blue!30, even odd rule] (1.6,\y) circle (0.32) (1.6,\y) circle (0.18);
        \draw[black, thick] (1.6,\y) circle (0.32);
    }

    \node at (0.3,-2) {\textbf{Case 1}};
\end{scope}

\draw[gray, thick] (3,-2.0) -- (3,2.0);

\begin{scope}[shift={(4.8,0)}]
    \begin{scope}[shift={(-1.0,0)}]
        \draw[fill=green!40, draw=black, thick, scale=0.35]
            plot[domain=0:360, smooth, samples=60]
            ({1.0*cos(\x) + 0.2*cos(2*\x)},
             {0.8*sin(\x) + 0.1*sin(2*\x)});
        \node[font=\small] at (0,0) {$D_1$};
    \end{scope}

    \foreach \y in {1.0, 0} {
        \fill[red!20] (1.6,\y) circle (0.18);
        \fill[blue!30, even odd rule] (1.6,\y) circle (0.32) (1.6,\y) circle (0.18);
        \draw[black, thick] (1.6,\y) circle (0.32);
    }

    \fill[red!20] (1.3,-1.3) circle (0.18);
    \fill[blue!30, even odd rule] (1.3,-1.3) circle (0.32) (1.3,-1.3) circle (0.18);
    \fill[gray!10, even odd rule] (1.3,-1.3) circle (0.46) (1.3,-1.3) circle (0.32);
    \fill[blue!30, even odd rule] (1.3,-1.3) circle (0.60) (1.3,-1.3) circle (0.46);
    \draw[black, thick] (1.3,-1.3) circle (0.60);

    \node at (0.3,-2) {\textbf{Case 2}};
\end{scope}

\draw[gray, thick] (7.8,-2.0) -- (7.8,2.0);

\begin{scope}[shift={(9.6,0)}]
    \begin{scope}[shift={(-1.0,0)}]
        \draw[fill=green!40, draw=black, thick, scale=0.35]
            plot[domain=0:360, smooth, samples=60]
            ({1.0*cos(\x) + 0.2*cos(2*\x)},
             {0.8*sin(\x) + 0.1*sin(2*\x)});
        \node[font=\small] at (0,0) {$D_1$};
    \end{scope}

    \fill[red!20] (1.6,0) circle (0.18);
    \fill[blue!30, even odd rule] (1.6,0) circle (0.32) (1.6,0) circle (0.18);
    \draw[black, thick] (1.6,0) circle (0.32);

    \fill[red!20] (1.3,-1.3) circle (0.18);
    \fill[blue!30, even odd rule] (1.3,-1.3) circle (0.32) (1.3,-1.3) circle (0.18);
    \fill[gray!10, even odd rule] (1.3,-1.3) circle (0.46) (1.3,-1.3) circle (0.32);
    \fill[blue!30, even odd rule] (1.3,-1.3) circle (0.60) (1.3,-1.3) circle (0.46);
    \draw[black, thick] (1.3,-1.3) circle (0.60);

    \fill[red!20] (1.3,1.5) circle (0.18);
    \fill[blue!30, even odd rule] (1.3,1.5) circle (0.32) (1.3,1.5) circle (0.18);
    \fill[gray!10, even odd rule] (1.3,1.5) circle (0.46) (1.3,1.5) circle (0.32);
    \fill[blue!30, even odd rule] (1.3,1.5) circle (0.60) (1.3,1.5) circle (0.46);
    \fill[gray!10, even odd rule] (1.3,1.5) circle (0.74) (1.3,1.5) circle (0.60);
    \fill[blue!30, even odd rule] (1.3,1.5) circle (0.88) (1.3,1.5) circle (0.74);
    \draw[black, thick] (1.3,1.5) circle (0.88);

    \node at (0.3,-2) {\textbf{Case 3}};
\end{scope}

\end{tikzpicture}
\caption{Schematic of the three multi-probe cases.}
\label{fig:multi_probe_setup}
\end{figure}
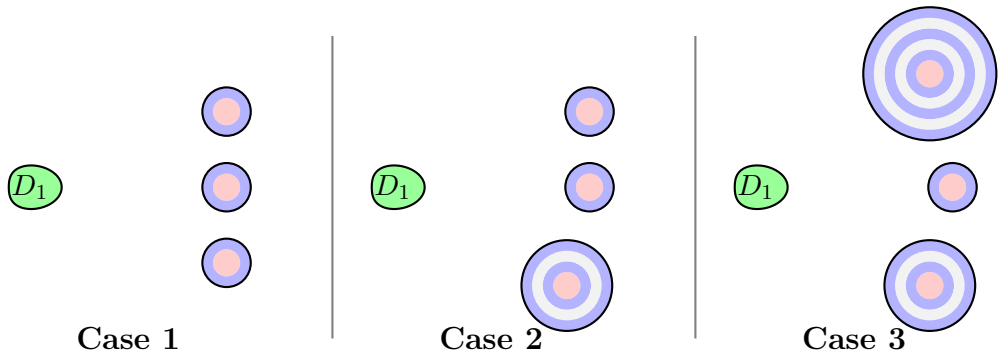

We now consider the case of multiple probes. Figure~\ref{fig:multi_probe_setup} shows the distribution of the probes in the three cases. Each probe has core radius $0.1$ and each shell has thickness $0.1$. The three probe configurations are as follows:
\begin{itemize}
    \item Case 1: three single-layer probes centered at $(1, -0.5)$, $(1, 0)$ and $(1, 0.5)$.
    \item Case 2: two single-layer probes centered at $(1, 0.5)$ and $(1, 0)$, and one two-layer probe centered at $(0.8, -0.7)$.
    \item Case 3: one single-layer probe centered at $(1, 0)$, one two-layer probe centered at $(0.8, -0.7)$, and one three-layer probe centered at $(0.8, 1)$.
\end{itemize}

The parameters for $D_{1}$ and the measurement setup are the same as those in the previous examples. For the three cases, we have numerically verified that $\mathcal{N}_2 \to 0$ as $\delta \to 0$, confirming that the reduced-domain method remains valid in these multi-probe settings.

\begin{figure}[htbp]
   \centering
   \subfloat[Case 1]
    {   \includegraphics[width=0.36\textwidth]{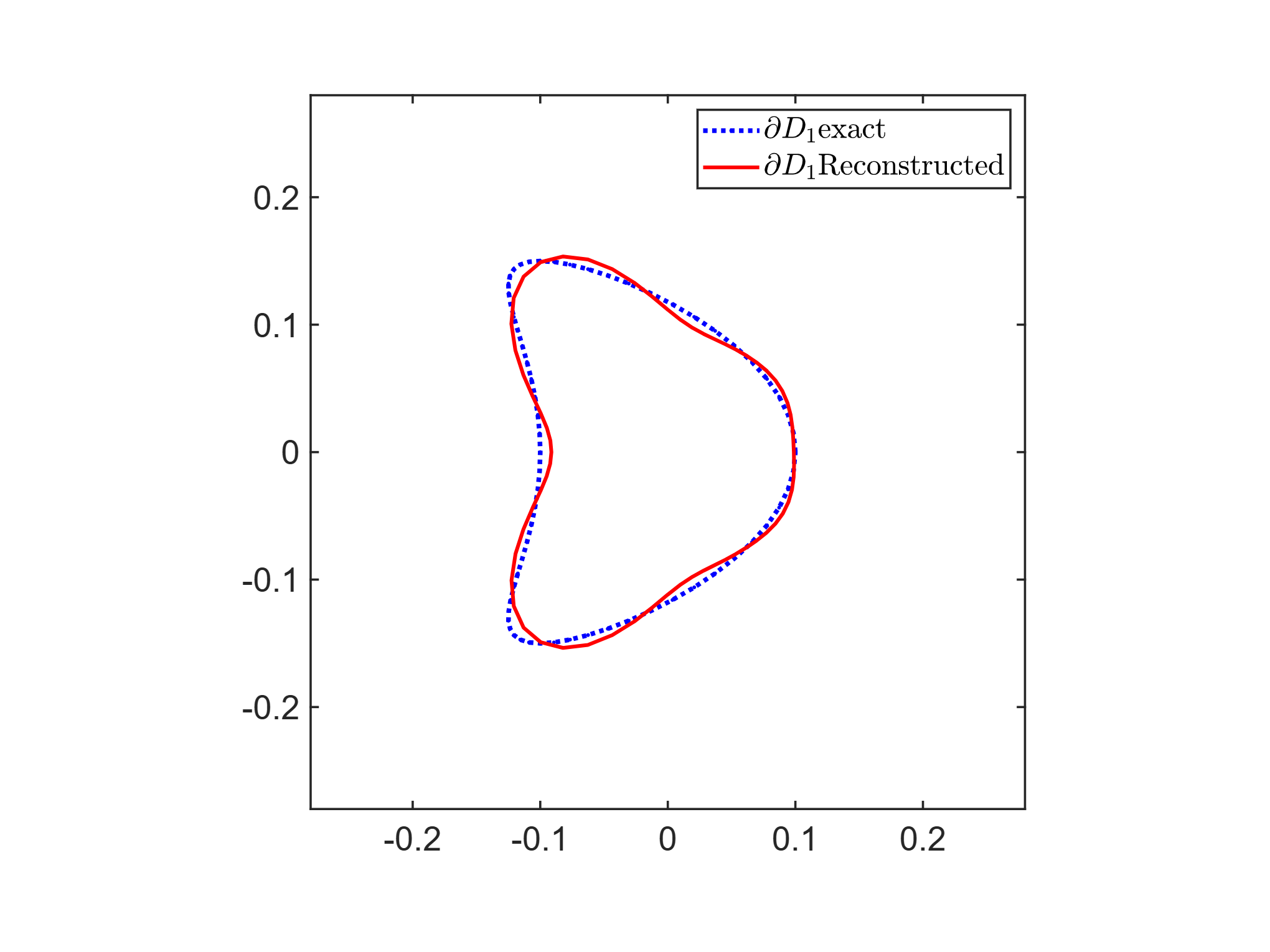}
     } \hspace{-1.4cm}
     \subfloat[Case 2]
     {   \includegraphics[width=0.36\textwidth]{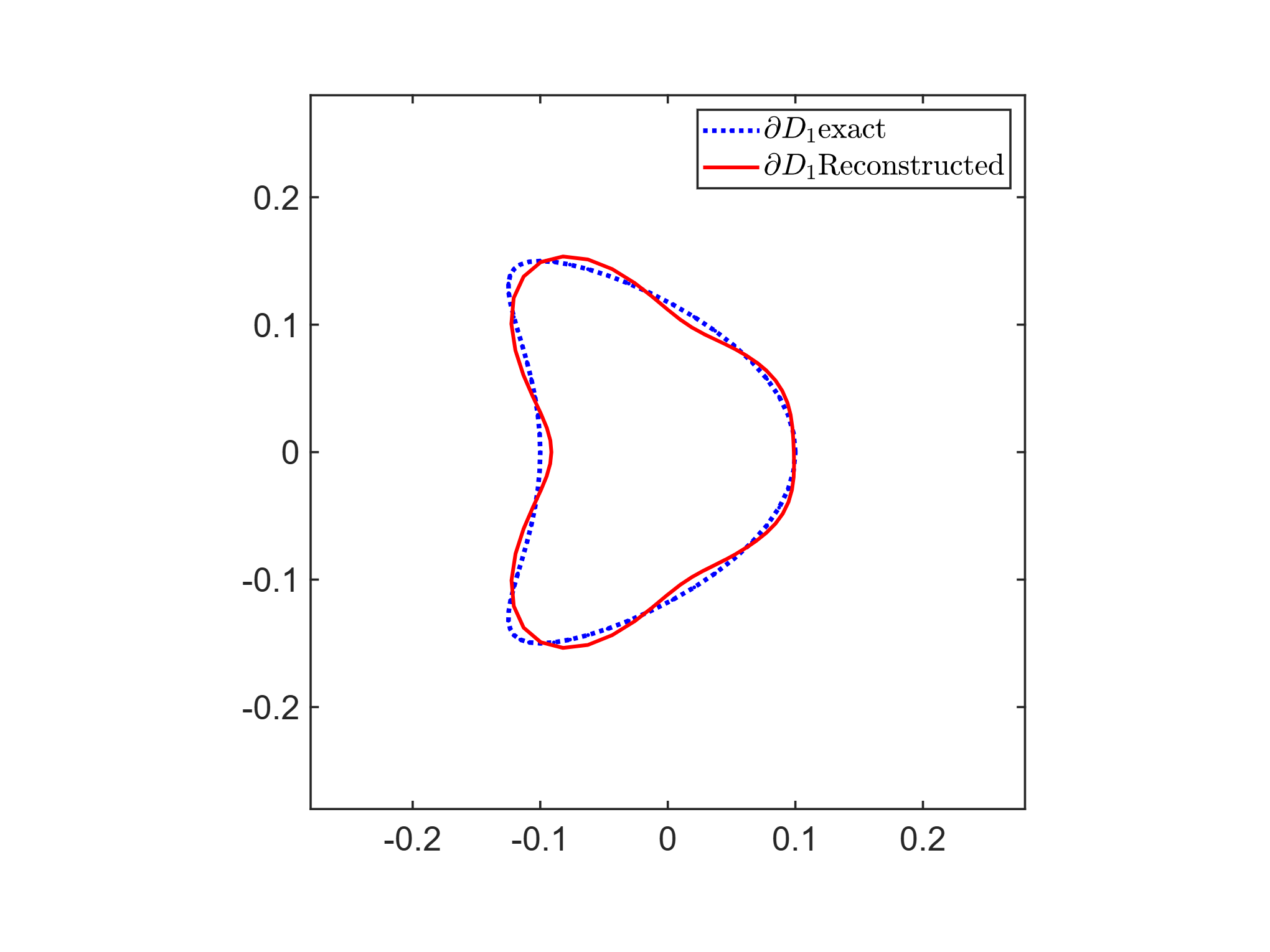}
     } \hspace{-1.4cm}
     \subfloat[Case 3]
     {   \includegraphics[width=0.36\textwidth]{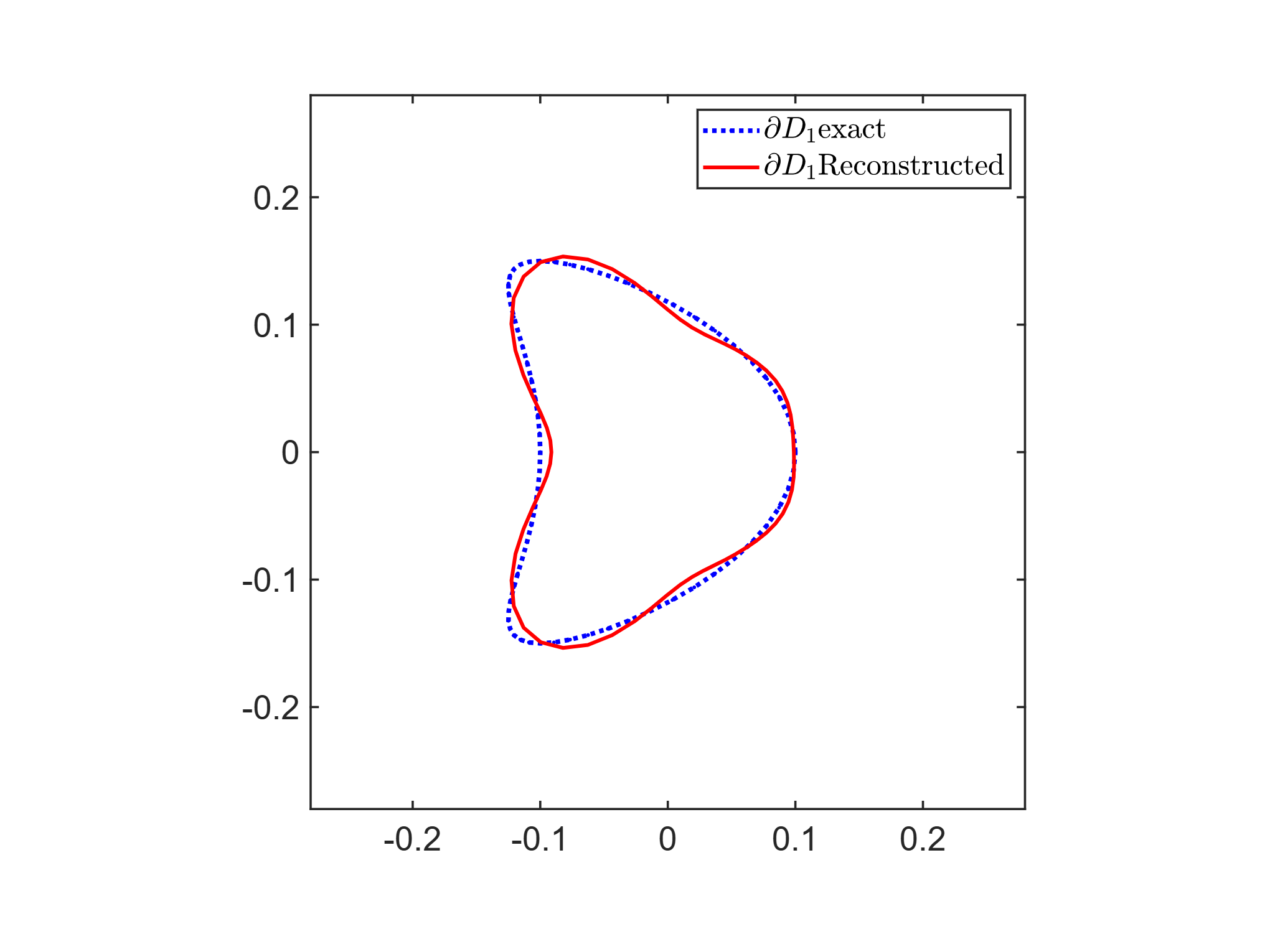}
     }
 \caption{Reconstruction results for the kite-shaped target using the reduced-domain method for the three multi-probe cases.}
 \label{fig:5}
\end{figure}

Figure~\ref{fig:5} presents the reconstruction results for the kite-shaped sample using the reduced-domain method for the three cases. In all cases, the method recovers $D_1$ with high accuracy, confirming its robustness to multiple probes regardless of layer composition. The computational efficiency of the two methods for these three cases is reported in Table~\ref{tab:multi_probe_efficiency}, with all inversion parameters kept the same as in the multi-layer examples.

\begin{table}[htbp]
\centering
\setlength{\tabcolsep}{18pt}
\caption{Computational time comparison for the three multi-probe cases (kite-shaped sample).}
\begin{tabular}{cccc}

\toprule
Case & $t_{\text{re}}$ (s) & $t_{\text{fu}}$ (s) & Speedup \\
\midrule
1 & 2.0449 & 48.6328 & 23.78\\

2 & 2.0218& 84.2347& 41.66\\

3 &2.0032 & 179.9859& 89.85\\
\bottomrule
\end{tabular}
\label{tab:multi_probe_efficiency}
\end{table}

As shown in the table, the speedup increases with probe complexity, from 23.78 in Case 1 to 41.66 in Case 2 and 89.85 in Case 3, confirming that the efficiency gain of the reduced-domain method becomes more pronounced for more complex probes. These results are consistent with the earlier single-probe and multi-layer findings, demonstrating that the computational advantage persists in multi-probe settings..

\section{Conclusion}
In this paper, we reformulate the cloaking problem for a core-shell probe as a minimization problem with respect to the permittivity of the coating shell, and prove the existence of a minimizer. By leveraging localized anomalous resonance, we show that the core-shell probe can be rendered nearly invisible as $\delta \to 0$. This property enables a reduced-domain method for reconstructing the sample $D_1$, which neglects the probe entirely during the shape imaging. Numerical experiments demonstrate that the cloaked probe significantly improves reconstruction accuracy in the small-probe regime, while in the large-probe regime, the reduced-domain method achieves substantial speedups over the full-domain approach with comparable accuracy, and exhibits good scalability with respect to probe complexity and number. Future work includes extensions to more general layered media and more complex probe geometries such as slit-type probes commonly used in NSOM systems.

	
	
	
	\addcontentsline{toc}{section}{References}

\end{document}